\documentclass[a4paper]{amsart}

\usepackage{amssymb, enumerate}
\usepackage{aliascnt}
\usepackage{hyperref}
\usepackage{tikz}
\usepackage{xy} \xyoption{all}

\newtheorem{lma}{Lemma}[section]

\newaliascnt{thmCt}{lma}
\newtheorem{thm}[thmCt]{Theorem}
\aliascntresetthe{thmCt}

\newaliascnt{corCt}{lma}
\newtheorem{cor}[corCt]{Corollary}
\aliascntresetthe{corCt}

\newaliascnt{prpCt}{lma}
\newtheorem{prp}[prpCt]{Proposition}
\aliascntresetthe{prpCt}

\theoremstyle{definition}

\newaliascnt{pgrCt}{lma}
\newtheorem{pgr}[pgrCt]{}
\aliascntresetthe{pgrCt}

\newaliascnt{dfnCt}{lma}
\newtheorem{dfn}[dfnCt]{Definition}
\aliascntresetthe{dfnCt}

\newaliascnt{rmkCt}{lma}
\newtheorem{rmk}[rmkCt]{Remark}
\aliascntresetthe{rmkCt}

\newaliascnt{rmksCt}{lma}

\aliascntresetthe{rmksCt}

\newaliascnt{exaCt}{lma}

\aliascntresetthe{exaCt}

\newaliascnt{qstCt}{lma}
\newtheorem{qst}[qstCt]{Question}
\aliascntresetthe{qstCt}

\newaliascnt{cnjCt}{lma}

\aliascntresetthe{cnjCt}

\newaliascnt{ntnCt}{lma}
\newtheorem{ntn}[ntnCt]{Notation}
\aliascntresetthe{ntnCt}

\newcounter{theoremintro}

\newtheorem{thmIntro}[theoremintro]{Theorem}
\newtheorem{exaIntro}[theoremintro]{Example}

\DeclareMathOperator{\sr}{sr}
\DeclareMathOperator{\rr}{rr}
\DeclareMathOperator{\gr}{gr}
\DeclareMathOperator{\gen}{gen}
\DeclareMathOperator{\Lg}{Lg}
\DeclareMathOperator{\Gen}{Gen}

\DeclareMathOperator{\locdim}{locdim}

\newcommand{\SubSep}{\mathrm{Sub}_{\mathrm{sep}}}
\newcommand{\vect}[1]{\mathbf{#1}}
\newcommand{\ca}{$C^*$-algebra}
\newcommand{\grPre}{\gr_0}
\newcommand{\freeVar}{\_\,}
\newcommand{\axiomD}[1]{(D#1)}

\newcommand{\sa}{\mathrm{sa}}

\title{The generator rank of $C^*$-algebras}
\author{Hannes Thiel}
\address{Hannes Thiel
Mathematisches Institut, Universit\"at M\"unster, Einsteinstr.~62, 48149 M\"unster, Germany.}
\email{hannes.thiel@posteo.de}
\urladdr{www.hannesthiel.org}

\thanks{The author was partially supported by the Danish National Research Foundation through the Centre for Symmetry and Deformation, and by the Deutsche Forschungsgemeinschaft (DFG, German Research Foundation) under the SFB 878 (Groups, Geometry \& Actions) and under Germany's Excellence Strategy EXC 2044-390685587 (Mathematics M\"{u}nster: Dynamics-Geometry-Structure).
}
\subjclass[2010]%
{Primary
46L05, 
46L85; 
Secondary
54F45, 
55M10. 
}
\date{\today}

\begin{document}

\begin{abstract}
We show that every AF-algebra is generated by a single operator.
This was previously unclear, since the invariant that assigns to a \ca{} its minimal number of generators lacks natural permanence properties.
In particular, it may increase when passing to ideals or inductive limits.

To obtain a better behaved theory, we not only ask if a \ca{} is generated by $n$ elements, but also if generating $n$-tuples are dense.
This defines the generator rank, which we show has many natural permanence properties:
it does not increase when passing to ideals, quotients or inductive limits.
\end{abstract}

\maketitle

\section{Introduction}

The generator problem for \ca{s} asks to determine which \ca{s} are singly generated.
More generally, for a given \ca{} $A$ one wants to compute the minimal number of generators. 
For a more detailed discussion of the generator problem, we refer to \cite{ThiWin14GenZStableCa}. 

Given a \ca{} $A$, let us denote by $\gen(A)$ the minimal number of \emph{self-adjoint} generators for $A$, and set $\gen(A)=\infty$ if $A$ is not finitely generated;
see \cite{Nag04SingleGenRnkCa}.
The restriction to self-adjoint elements is mainly for convenience.
It only leads to a minor variation of the original generator problem, since two self-adjoint elements $a$ and $b$ generate the same sub-\ca{} as the element $a+ib$.
In particular, $A$ is singly generated if and only if $\gen(A)\leq 2$.
For a compact, metric space $X$, it is easy to see that $\gen(C(X))\leq k$ if and only if $X$ can be embedded into $\mathbb{R}^k$.

It is also easy to see that $\gen(F)\leq 2$ for every finite-dimensional \ca{}~$F$.
However, this does not readily show that every AF-algebras is singly generated, since the minimal number of (self-adjoint) generators may increase when passing to inductive limits:

\begin{exaIntro}
Let $X\subseteq\mathbb{R}^2$ be the topologists sine-curve given by:
\[
X=\{0\}\times[-1,1] \cup \big\{ (t,\sin(t^{-1})) : t\in(0,1/2\pi] \big\}.
\]
Then $X$ can be embedded into $\mathbb{R}^2$ but not into $\mathbb{R}^1$, and therefore $\gen(C(X))=2$.
However, $X$ is an inverse limit of spaces $X_n$ that are each homeomorphic to the interval. 
Therefore $C(X)\cong\varinjlim_n C(X_n)$, with $\gen(C(X))=2$, while one the other hand $\gen(C(X_n))=\gen(C([0,1]))=1$ for all $n$.
The spaces $X$ and $X_1,X_2,X_3$ are shown below.
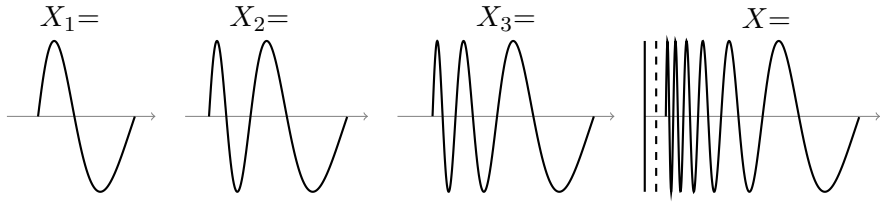
\[
\makebox{
\begin{tikzpicture}[xscale=30]
\draw [help lines, ->] (0.05,0) -- (0.115,0);
\draw [thick, domain=1/(7*pi):1/(3*pi), samples=500] plot (\x, {sin((1/\x) r)});
\node [right] at (0.06,1.3) {$X_1$=};
\end{tikzpicture}
}
\makebox{
\begin{tikzpicture}[xscale=30]
\draw [help lines, ->] (0.035,0) -- (0.115,0);
\draw [thick, domain=1/(9*pi):1/(3*pi), samples=500] plot (\x, {sin((1/\x) r)});
\node [right] at (0.05,1.3) {$X_2$=};
\end{tikzpicture}
}
\makebox{
\begin{tikzpicture}[xscale=30]
\draw [help lines, ->] (0.02,0) -- (0.115,0);
\draw [thick, domain=1/(11*pi):1/(3*pi), samples=500] plot (\x, {sin((1/\x) r)});
\node [right] at (0.05,1.3) {$X_3$=};
\end{tikzpicture}
}
\makebox{
\begin{tikzpicture}[xscale=30]
\draw [help lines, ->] (0.012,0) -- (0.115,0);
\draw [thick] (0.017,-1) -- (0.017,1);
\draw [dotted, thick] (0.020,-1) -- (0.020,1);
\draw [thick, domain=1/(15*pi):1/(3*pi), samples=500] plot (\x, {sin((1/\x) r)});
\node [right] at (0.05,1.3) {$X$=};
\end{tikzpicture}
}
\]

${}$

By considering the spaces $X\times[0,1]$ and $X_n\times[0,1]$, one obtains a sequence of singly generated \ca{s} whose inductive limit is not singly generated.
\end{exaIntro}

To get a better behaved theory, instead of counting the minimal number of generators, we count the minimal number of `stable' generators. 
More precisely, given a unital, separable \ca{} $A$, let $\Gen_n(A)_\sa\subseteq A^n_\sa$ be the subset of self-adjoint $n$-tuples that generate $A$;
see \autoref{pgr:notation_gen}.
Then $A$ has generator rank at most $n$, denoted by $\gr(A)\leq n$, if $\Gen_{n+1}(A)_\sa$ is dense in $A^{n+1}_\sa$;
see \autoref{prp:grSep}.
The index shift follows the usual convention in (noncommutative) dimension theory:
For instance, a topological space has covering dimension at most~$n$ if every finite cover can be refined by a cover with index $n+1$;
similarly, a unital \ca{} has real rank at most $n$ if every tuple of $n+1$ self-adjoint elements can be approximated by a selfadjoint tuple that generates the \ca{} as a left ideal (see \autoref{pgr:rr_sr}). 
Using the index shift also leads to simpler formulas that involve the generator rank.
For example, by \autoref{prp:rr_less_gr}, we always have $\rr(A)\leq\gr(A)$.

To handle nonunital \ca{s}, we first develop a precursor to the generator rank, denoted $\grPre$, and then define $\gr(A):=\grPre(\widetilde{A})$, where $\widetilde{A}$ is the minimal unitization of $A$.
To also handle nonseparable \ca{s}, we consider a relative version of generation:
given an element $c$, we ask if a given self-adjoint tuple $\vect{a}$ can be approximated by tuples $\vect{b}$ such that $c$ is (approximately) contained in the sub-\ca{} generated by $\vect{b}$;
see \autoref{dfn:grPre}.

A \ca{} has generator rank zero if and only if it is commutative with totally disconnected spectrum;
see \autoref{prp:charGr0}.
A unital, separable \ca{} $A$ has generator rank at most one if and only if a generic element of $A$ is a generator;
see \autoref{rmk:gr1_means_generic}.
We compute the generator rank of commutative \ca{s}:

\begin{thmIntro}[\ref{prp:gr_commutative}]
Let $X$ be a locally compact, Hausdorff space.
Then
\[
\gr(C_0(X))=\locdim(X\times X),
\]
where $\locdim$ denotes the local dimension, which is defined as the supremum of the covering dimension of all compact subsets (see \autoref{pgr:locdim}).
\end{thmIntro}

For unital, separable \ca{s}, `$\gen(A)\leq n+1$' records that $\Gen_{n+1}(A)_\sa$ is nonempty, while `$\gr(A)\leq n$' records that $\Gen_{n+1}(A)_\sa$ is dense.
Thus, the generator rank is often much larger than the minimal number of self-adjoint generators.
The payoff, however, is that the generator rank has nicer permanence properties:

\begin{thmIntro}[\ref{prp:gr_idealQuotExt}, \ref{prp:gr_approx_limits}]
Let $A$ be a \ca{} and let $I\subseteq A$ be a closed, two-sided ideal.
Then
\[
\max\big\{ \gr(I),\gr(A/I) \big\} \leq \gr(A) \leq \gr(I)+\gr(A/I)+1.
\]
Further, if $A=\varinjlim_\lambda A_\lambda$ is an inductive limit, then
\[
\gr(A) \leq \liminf_\lambda\gr(A_\lambda).
\]
\end{thmIntro}

By showing that finite-dimensional \ca{s} have generator rank at most one, we deduce:

\begin{thmIntro}[\ref{prp:gr_AF-alg}]
Let $A$ be a separable AF-algebra.
Then $\gr(A)\leq 1$, and so a generic element of $A$ is a generator.
In particular, $A$ is singly generated.
\end{thmIntro}

In subsequent work, \cite{Thi20arX:grSubhom}, we compute the generator rank of subhomogeneous \ca{s}.
We obtain in particular that every $\mathcal{Z}$-stable approximately subhomogeneous (ASH) \ca{} has generator rank one.
In further work, \cite{Thi20arX:grZstableRR0}, we show that $\mathcal{Z}$-stable \ca{s} of real rank zero have generator rank one.
It follows that every classifiable, simple \ca{} has generator rank one.

\subsection*{Acknowledgments}

The author thanks James Gabe and Mikael R{\o}rdam for valuable comments and feedback.
I also want to thank the anonymous referee whose suggestions helped to greatly improve the paper.

This paper grew out of joint work with Karen Strung, Aaron Tikuisis, Joav Orovitz and Stuart White that started at the workshop `Set theory and \ca{s}' at the AIM in Palo Alto, January 2012.
In particular, the key \autoref{prp:building_generators} was obtained in that joint work.
The author benefited from many fruitful discussions with Strung, Tikuisis, Orovitz and White, and he wants to thank them for their support.

\subsection*{Notation}

We set $\mathbb{N}:=\{0,1,2\ldots\}$.
Given a \ca{} $A$, we use $A_\sa$ and $A_+$ to denote the set of self-adjoint and positive elements in $A$, respectively.
We denote by $\widetilde{A}$ and $A^+$ the minimal and forced unitization of $A$, respectively.
By an ideal in a \ca{} we always mean a closed, two-sided ideal.

Given $a,b\in A$, and $\varepsilon>0$, we write $a=_\varepsilon b$ if $\|a-b\|<\varepsilon$.
Given $a\in A$ and $G\subseteq A$, we write $a\in_\varepsilon G$ if there exists $b\in G$ with $a=_\varepsilon b$.
We use bold letters to denote tuples of elements, for example $\vect{a}=(a_1,\ldots,a_n)\in A^n$.
Given $\vect{a},\vect{b}\in A^n$, we write $\vect{a}=_\varepsilon\vect{b}$ if $a_j=_\varepsilon b_j$ for $j=1,\ldots,n$.
We denote by $C^*(\vect{a})$ the sub-\ca{} of $A$ generated by the elements of $\vect{a}$.
We write $A^n_\sa$ for $(A_\sa)^n$, the space of $n$-tuples of self-adjoint elements.

By an \emph{nc-polynomial} we mean a polynomial in noncommuting variables.

\section{A precursor of the generator rank}
\label{sec:grPre}

In this section, we introduce the invariant $\grPre$ for \ca{s} and we show that it behaves well when passing to ideals, quotients, inductive limits and extensions.
In \autoref{sec:gr}, we define the generator rank of a \ca{} $A$ as $\grPre(\widetilde{A})$ and in \autoref{sec:permanence} we will see that $\gr$ satisfies the same permanence properties as $\grPre$.

\begin{dfn}
\label{dfn:grPre}
Let $A$ be a \ca{}.
We define $\grPre(A)$ as the smallest integer $n\geq 0$ such that for every $a_0,\ldots,a_n\in A_\sa$, $\varepsilon>0$ and $c\in A$, there exist $b_0,\ldots,b_n\in A_\sa$ such that
\[
\|b_j-a_j\|<\varepsilon \text{ for } j=0,\ldots,n, \quad\text{ and }\quad
c\in_\varepsilon C^*(b_0,\ldots,b_n).
\] 
If no such $n$ exists, we set $\grPre(A)=\infty$.
\end{dfn}

\begin{prp}
\label{prp:gr0_quotients}
Let $A$ be a \ca{} and let $I\subseteq A$ be a (closed, two-sided) ideal.
Then $\grPre(A/I)\leq\grPre(A)$.
\end{prp}
\begin{proof}
Set $n:=\grPre(A)$, which we may assume to be finite.
Let $\pi\colon A\to A/I$ denote the quotient map.
It naturally induces a surjective map $A^{n+1}\to(A/I)^{n+1}$, which we also denote by $\pi$.
To verify $\grPre(A/I)\leq n$, let $\vect{a}\in (A/I)^{n+1}_\sa$, $\varepsilon>0$ and $c\in A/I$.
Lift $\vect{a}$ and $c$ to obtain $\vect{x}\in A^{n+1}_\sa$ and $z\in A$ such that
\[
\pi(\vect{x})=\vect{a}, \quad\text{ and }\quad \pi(z)=c.
\]
Using that $\grPre(A)\leq n$, we obtain $\vect{y}\in A^{n+1}_\sa$ such that
\[
\vect{y}=_\varepsilon\vect{x}, \quad\text{ and }\quad
z\in_\varepsilon C^*(\vect{y}).
\]
Set $\vect{b}:=\pi(\vect{y})$.
Then $\vect{b}=_\varepsilon\vect{a}$, and $c\in_\varepsilon C^*(\vect{b})$, as desired.
\end{proof}

For the next result, recall that a collection $(A_\lambda)_{\lambda\in\Lambda}$ of sub-\ca{s} of $A$ is said to \emph{approximate} $A$ if for every finite subset $F\subseteq A$ and for every $\varepsilon>0$, there exists $\lambda\in\Lambda$ such that $a\in_\varepsilon A_\lambda$ for every $a\in F$.

\begin{prp}
\label{prp:gr0_approx}
Let $A$ be a \ca{} and $n\in\mathbb{N}$.
Assume that $A$ is approximated by sub-\ca{s} $A_\lambda\subseteq A$ with $\grPre(A_\lambda)\leq n$ for each $\lambda$.
Then $\grPre(A)\leq n$.
\end{prp}
\begin{proof}
To verify $\grPre(A)\leq n$, let $\vect{a}\in A^{n+1}_\sa$, $\varepsilon>0$ and $c\in A$.
Since the $A_\lambda$ approximate $A$, there exist $\lambda$, a tuple $\vect{x}\in (A_\lambda)^{n+1}_\sa$, and $z\in A_\lambda$ such that
\[
\vect{x}=_{\varepsilon/2}\vect{a}, \quad\text{ and }\quad
c=_{\varepsilon/2} z.
\]
Using that $\grPre(A_\lambda)\leq n$, we obtain $\vect{b}\in(A_\lambda)^{n+1}_\sa$ such that
\[
\vect{b}=_{\varepsilon/2}\vect{x}, \quad\text{ and }\quad
z\in_{\varepsilon/2} C^*(\vect{b}).
\]
Then $\vect{b}=_\varepsilon\vect{a}$ and $c\in_\varepsilon C^*(\vect{b})$, as desired.
\end{proof}

\begin{prp}
\label{prp:gr0_limits}
Let $A=\varinjlim_{\lambda\in\Lambda} A_\lambda$ be an inductive limit.
Then $\grPre(A)\leq\liminf_{\lambda} \grPre(A_\lambda)$.
\end{prp}
\begin{proof}
For each $\lambda\in\Lambda$, let $B_\lambda$ be the image of $A_\lambda$ in the inductive limit $A$.
Then $B_\lambda$ is a quotient of $A_\lambda$, and therefore $\grPre(B_\lambda)\leq\grPre(A_\lambda)$, by \autoref{prp:gr0_quotients}.

Let $\Lambda_0\subseteq\Lambda$ be a cofinal subset.
Then $A$ is approximated by the collection $(B_\lambda)_{\lambda\in\Lambda_0}$.
By \autoref{prp:gr0_approx}, we get
\[
\grPre(A)\leq \sup_{\lambda\in\Lambda_0} \grPre(B_\lambda) \leq \sup_{\lambda\in\Lambda_0} \grPre(A_\lambda).
\]
Since this holds for every cofinal subset $\Lambda_0\subseteq\Lambda$, we obtain:
\[
\grPre(A)
\leq \inf \big\{ \sup_{\lambda\in\Lambda_0} \grPre(A_\lambda) : \Lambda_0\subseteq\Lambda \text{ cofinal} \big\}
= \liminf_{\lambda\in\Lambda} \grPre(A_\lambda). \qedhere
\]
\end{proof}

Next, we will show that $\grPre$ does not increase when passing to ideals.
Since the proof is technically involved, we first prove two preparatory lemmas.
Recall that an \emph{nc-polynomial} is a polynomial in noncommuting variables.
We will say that a nc-polynomial $p$ (in $n+1$ variables) is homogeneous of degree $d$ if $q(ta_0,\ldots,ta_n)=t^d q(a_0,\ldots,a_n)$ for all $t\in\mathbb{R}_+$ and all elements $a_0,\ldots,a_n$ (in some \ca).
We will repeatedly use that every ideal in a \ca{} contains a quasi-central approximate unit;
see \cite[Proposition~II.4.3.2, p.82]{Bla06OpAlgs}.

\begin{lma}
\label{prp:gr0_ideals_lemma1}
Let $A$ be a \ca{}, let $I\subseteq A$ be an ideal, let $(h_\lambda)_\lambda$ be a quasi-central approximate unit for $I$, and let $p$ be a nc-polynomial that is homogeneous of degree $d$.
Then 
\[
\lim_\lambda \left\| h_\lambda^d\cdot p(a_0,\ldots,a_n) 
- p(h_\lambda^{1/2} a_0 h_\lambda^{1/2}, \ldots, h_\lambda^{1/2} a_n h_\lambda^{1/2}) \right\| = 0,
\]
for all $a_0,\ldots,a_n\in A$.
\end{lma}
\begin{proof}
Using that $(h_\lambda)_\lambda$ is quasi-central and that $p$ has degree $d$, we get
\[
\lim_\lambda \left\| h_\lambda^d\cdot p(a_0,\ldots,a_n) - p(h_\lambda a_0,\ldots,h_\lambda a_n) \right\| = 0.
\]
The result follows once we show that $\lim_\lambda \| h_\lambda a - h_\lambda^{1/2} a h_\lambda^{1/2} \| = 0$ for every $a\in A$.
To prove this, let $a\in A$ and $\varepsilon>0$.
We may assume that $a\neq 0$.
Use the Stone-Weierstrass theorem to obtain a polynomial $q$ such that $\|h^{1/2}-q(h)\|<\tfrac{\varepsilon}{3\|a\|}$ for every contractive, positive element $h$ in $A$.
Using that $(h_\lambda)_\lambda$ is quasi-central, choose $\lambda_0\in\Lambda$ such that $q(h_\lambda)a=_{\varepsilon/3} aq(h_\lambda)$ for all $\lambda\geq\lambda_0$.
Then
\[
h_\lambda^{1/2} a 
=_{\varepsilon/3} q(h_\lambda) a 
=_{\varepsilon/3} aq(h_\lambda) 
=_{\varepsilon/3} ah_\lambda^{1/2}
\]
for all $\lambda\geq\lambda_0$, as desired.
\end{proof}

In the next results, we set $|\vect{a}| :=\sum_{j=0}^n |a_j| = \sum_{j=0}^n (a_j^*a_j)^{1/2}$ for $\vect{a}\in A^{n+1}$.

\begin{lma}
\label{prp:gr0_ideals_lemma2}
Let $A$ be a \ca{}, let $I\subseteq A$ be an ideal, and let $n\in\mathbb{N}$ such that $\grPre(A)\leq n$.
Then for every $\vect{x}\in I^{n+1}_\sa$, $\varepsilon>0$ and $w\in I$, there exist $\vect{y}\in I^{n+1}_\sa, \delta>0$ and a nc-polynomial $p$ with vanishing zero coefficient such that
\[
\vect{y}=_{\varepsilon}\vect{x}, \quad\text{ and }\quad
w=_{\varepsilon} p(\vect{y}) (|\vect{y}|-\delta)_+ w.
\]
\end{lma}
\begin{proof}
Let $(h_\lambda)_{\lambda\in\Lambda}$ be a positive, contractive approximate unit for $I$ that is quasi-central with respect to $A$.
Given $\vect{a}=(a_0,\ldots,a_n)\in A^{n+1}$ and $\lambda\in\Lambda$, we set
\begin{align*}
\vect{a}^{(\lambda)}
&:= (h_\lambda^{1/2}a_0h_\lambda^{1/2},\ldots,h_\lambda^{1/2}a_nh_\lambda^{1/2}).
\end{align*}

\emph{Claim~1:
Let $\vect{a}\in A^{n+1}_\sa, w\in I, \varepsilon>0, \lambda_0\in\Lambda$, and let $p$ be a nc-polynomial with vanishing zero coefficient.
Then there exists $\lambda\geq\lambda_0$ such that}
\[
p(\vect{a})w =_\varepsilon p(\vect{a}^{(\lambda)})w.
\]

To prove the claim, decompose $p$ as $p=\sum_{d=1}^D p_d$, where $p_d$ is homogeneous of degree $d$.
Without loss of generality, we may assume that $w\neq 0$.
For each $d\in\{1,\ldots,D\}$, the net $(h_\lambda^d)_\lambda$ is an approximate unit for $I$;
see \cite[Proposition~II.4.1.2, p.79]{Bla06OpAlgs}.
Hence, we can choose $\lambda_1\geq\lambda_0$ such that
\[
\left\| p_d(\vect{a})w - p_d(\vect{a}) h_\lambda^d w \right\| < \frac{\varepsilon}{2D} 
\]
for all $d\in\{1,\ldots,D\}$ and $\lambda\geq\lambda_1$.
Using \autoref{prp:gr0_ideals_lemma1}, choose $\lambda\geq\lambda_1$ such that
\[
\left\| h_\lambda^d p_d(\vect{a}) - p_d(\vect{a}^{(\lambda)}) \right\| 
< \frac{\varepsilon}{2D\|w\|} \ \text{ for } d=1,\ldots,D. 
\]
Then
\begin{align*}
\left\| p_d(\vect{a}) w - p_d(\vect{a}^{(\lambda)}) w \right\|
&\leq \left\| p_d(\vect{a}) w - p_d(\vect{a}) h_\lambda^d w \right\| 
+ \left\| p_d(\vect{a}) h_\lambda^d w - p_d(\vect{a}^{(\lambda)}) w \right\| \\
&< \frac{\varepsilon}{2D} + \|w\| \frac{\varepsilon}{2D\|w\|} =  \frac{\varepsilon}{D}.
\end{align*}
It follows that $\| p(\vect{a}) w - p(\vect{a}^{(\lambda)}) w \| < \varepsilon$, which proves the claim.

\emph{Claim~2:
Let $\vect{y}\in I^{n+1}_\sa, \varepsilon>0$, and let $p_0$ be a nc-polynomial with vanishing zero coefficient.
Then there exists $\delta>0$ and a nc-polynomial $p_1$ such that}
\[
p_0(\vect{y}) =_\varepsilon p_0(\vect{y})p_1(\vect{y})(|\vect{y}|-\delta)_+.
\]

To prove the claim, note that $y_j\in\overline{I|\vect{y}|}$ for each $j\in\{0,\ldots,n\}$.
It follows that $p_0(\vect{y})\in\overline{I|\vect{y}|}$.
Hence
\[
\lim_{k\to\infty} \left\| p_0(\vect{y}) - p_0(\vect{y})|\vect{y}|^{1/k} \right\| = 0,
\]
which allows us to choose $k$ such that $p_0(\vect{y})=_{\varepsilon/2}p_0(\vect{y})|\vect{y}|^{1/k}$.
Using the Stone-Weierstrass theorem, choose a polynomial $p_1'$ such that $|\vect{y}|^{1/k} =_{\varepsilon/2\|p_0(\vect{y})\|} p_1'(|\vect{y}|)|\vect{y}|$.
Then
\[
p_0(\vect{y})
=_{\varepsilon/2} p_0(\vect{y})|\vect{y}|^{1/k}
=_{\varepsilon/2} p_0(\vect{y}) p_1'(|\vect{y}|)|\vect{y}|.
\]
Next, choose $\delta>0$ such that $(|\vect{y}|-\delta)_+$ is very close to $|\vect{y}|$, and choose a nc-polynomial $p_1$ such that $p_1(\vect{y})$ is very close to $p_1'(|\vect{y}|)|$, such that
\[
p_0(\vect{y}) =_\varepsilon p_0(\vect{y}) p_1(\vect{y}) (|\vect{y}|-\delta)_+.
\]
This proves the claim.

To prove the Lemma, let $\vect{x}\in I^{n+1}_\sa$, $\varepsilon>0$ and $w\in I$. 
We first choose $e\in I_+$ such that $w=_{\varepsilon/4} ew$.
Using that $\grPre(A)\leq n$, we obtain $\vect{a}\in A^{n+1}_\sa$ such that
\[
\vect{a}=_{\varepsilon/2}\vect{x}, \quad\text{ and }\quad
e\in_{\varepsilon/4\|w\|}C^*(\vect{a}).
\]
Choose a nc-polynomial $p_0$ with vanishing zero coefficient such that $e=_{\varepsilon/4\|w\|} p_0(\vect{a})$.
Let $\lambda_0$ such that $\vect{x}=_{\varepsilon/4}\vect{x}^{(\lambda)}$ for all $\lambda\geq\lambda_0$.
Using Claim~1, we obtain $\lambda\geq\lambda_0$ such that
\[
p_0(\vect{a})w =_{\varepsilon/4} p_0(\vect{a}^{(\lambda)})w.
\]
Set $\vect{y}:=\vect{a}^{(\lambda)} \in I^{n+1}_\sa$.
Since $\vect{a}=_{\varepsilon/2}\vect{x}$, we have $\vect{a}^{(\lambda)}=_{\varepsilon/2}\vect{x}^{(\lambda)}$.
Thus,
\[
\vect{y}
= \vect{a}^{(\lambda)}
=_{\varepsilon/2} \vect{x}^{(\lambda)}
=_{\varepsilon/2} \vect{x},
\]
and so $\vect{y}=_{\varepsilon}\vect{x}$.
Further
\begin{align*}
\big\| p_0(\vect{y})w - w \big\| 
&\leq \big\| p_0(\vect{a}^{(\lambda)})w - p_0(\vect{a})w \big\| 
+ \big\| p_0(\vect{a})w - ew \big\| + \big\| ew-w \big\| \\
&< \frac{\varepsilon}{4} + \|w\| \frac{\varepsilon}{4\|w\|} + \frac{\varepsilon}{4} 
= \frac{3}{4}\varepsilon.
\end{align*}
Using Claim~2, we obtain $\delta>0$ and a nc-polynomial $p_1$ such that
\[
p_0(\vect{y}) =_{\varepsilon/4\|w\|} p_0(\vect{y})p_1(\vect{y})(|\vect{y}|-\delta)_+.
\]
Let $p$ be the product of $p_0$ and $p_1$.
Then
\begin{align*}
\big\| p(\vect{y}) (|\vect{y}|-\delta)_+ w - w \big\|
&\leq \big\| p_0(\vect{y})p_1(\vect{y})(|\vect{y}|-\delta)_+w - p_0(\vect{y})w \big\| 
+ \big\| p_0(\vect{y})w - w \big\| \\
&< \|w\|\frac{\varepsilon}{4\|w\|} + \frac{3}{4}\varepsilon = \varepsilon,
\end{align*}
as desired.
\end{proof}

\begin{prp}
\label{prp:gr0_ideals}
Let $A$ be a \ca{}, and let $I\subseteq A$ be a (closed, two-sided) ideal.
Then $\grPre(I)\leq\grPre(A)$.
\end{prp}
\begin{proof}
Set $n:=\grPre(A)$, which we may assume to be finite.
Let $(h_\lambda)_{\lambda\in\Lambda}$ be a positive, contractive approximate unit for $I$ that is quasi-central with respect to $A$.
As in the proof of \autoref{prp:gr0_ideals_lemma2}, we set 
\begin{align*}
\vect{a}^{(\lambda)}
&:= (h_\lambda^{1/2}a_0h_\lambda^{1/2},\ldots,h_\lambda^{1/2}a_nh_\lambda^{1/2})
\end{align*}
for $\vect{a}=(a_0,\ldots,a_n)\in A^{n+1}$ and $\lambda\in\Lambda$.

To verify $\grPre(I)\leq n$, let $\vect{x}\in I^{n+1}_\sa$, $\varepsilon>0$ and $w\in I$.
We need to find $\vect{z}\in I^{n+1}_\sa$ such that
\[
\vect{z}=_\varepsilon\vect{x},\quad\text{ and }\quad
w\in_\varepsilon C^*(\vect{z}).
\]
Without loss of generality, we may assume that $w\neq 0$.
We construct $\vect{z}$ in 3 steps.

\emph{Step~1:}
Apply \autoref{prp:gr0_ideals_lemma2} to obtain $\vect{y}\in I^{n+1}_\sa, \delta>0$ and a nc-polynomial $p$ with vanishing zero coefficient such that
\begin{align}
\label{eq:gr0_ideals:Step1A}
\vect{y}=_{\varepsilon/2}\vect{x}, \quad\text{ and }\quad
w=_{\varepsilon/3} p(\vect{y}) (|\vect{y}|-\delta)_+ w.
\end{align}
Set
\[
M := \max \big\{ 2\|p(\vect{y})\|, 2\|(|\vect{y}|-\delta)_+\| \big\}.
\]
Choose $\eta>0$ so small such that for all $\vect{b}\in A^{n+1}_\sa$ with $\vect{b}=_\eta\vect{y}$ we have
\begin{align}
\label{eq:gr0_ideals:Step1B}
|\vect{b}| =_\delta |\vect{y}|,
\end{align}
and such that for all $\vect{z}\in I^{n+1}_\sa$ with $\vect{z}=_\eta\vect{y}$ we have
\begin{align}
\label{eq:gr0_ideals:Step1C}
w =_{\varepsilon/3} p(\vect{z})\cdot (|\vect{z}|-\delta)_+ w, \quad
\|p(\vect{z})\| \leq M, \quad
\|(|\vect{z}|-\delta)_+\| \leq M.
\end{align}
We may assume that $\eta\leq\tfrac{\varepsilon}{2}$ and $\eta\leq\delta$.

Using that $\grPre(A)\leq n$, we obtain $\vect{b}\in A^{n+1}_\sa$ and a nc-polynomial $q$ such that
\begin{align}
\label{eq:gr0_ideals:Step1D}
\vect{b}=_\eta\vect{y}, \quad\text{ and }\quad
w=_{\varepsilon/3M^2} q(\vect{b}).
\end{align}

\emph{Step~2:
We will show that}
\begin{align}
\label{eq:gr0_ideals:Step2a}
\lim_\lambda \left\| (|\vect{b}^{(\lambda)}|-\delta)_+ q(\vect{b}) 
- (|\vect{b}^{(\lambda)}|-\delta)_+ q(\vect{b}^{(\lambda)}) \right\| = 0.
\end{align}

We consider the product $\prod_{\lambda\in\Lambda} A$.
Set
\[
\bigoplus_\lambda A := \big\{ (a_\lambda)_{\lambda\in\Lambda} : \lim_\lambda\|a_\lambda\|=0 \big\},
\]
which is an ideal in  $\prod_{\lambda\in\Lambda} A$.
Then set $Q:=\prod_\lambda A / \bigoplus_\lambda A$.
Given $(a_\lambda)_\lambda\in\prod_{\lambda\in\Lambda} A$, we let $\langle (a_\lambda)_\lambda\rangle$ denote the image in $Q$.
For an element $a\in A$, we write $\langle a\rangle$ instead of $\langle (a)_\lambda\rangle$ for the image of the constant tuple.
Set $h:=\langle (h_\lambda)_\lambda\rangle\in Q$.
Since $(h_\lambda)_\lambda$ is quasi-central, $h$ commutes with $\langle a\rangle$, for every $a\in A$.

Let $d\geq 1$.
We have $\vect{b}=_\eta \vect{y}$ (see \eqref{eq:gr0_ideals:Step1D}), and thus $|\vect{b}|=_\delta|\vect{y}|$ by \eqref{eq:gr0_ideals:Step1B}.
Let $\pi\colon A\to A/I$ denote the quotient map.
Since $\vect{y}$ belongs to $I^{n+1}_\sa$, we have $\pi(|\vect{y}|)=0$.
Hence, $\pi(|\vect{b}|-\delta)\leq 0$ and so $\pi((|\vect{b}|-\delta)_+)=0$.
It follows that $\langle (|\vect{b}|-\delta)_+ \rangle \cdot (1-h^d) = 0$.

For each $j\in\{0,\ldots,n\}$, we have 
\begin{align*}
\left\langle \big( |b_j^{(\lambda)}| \big)_\lambda \right\rangle
&= \left\langle \big( ( h_\lambda^{1/2} b_j h_\lambda b_j h_\lambda^{1/2} )^{1/2} \big)_\lambda \right\rangle 
= \big( h^{1/2} \langle b_j\rangle h \langle b_j\rangle h^{1/2} \big)^{1/2} \\
&= \big( \langle b_j\rangle h^2 \langle b_j\rangle  \big)^{1/2}
\leq \big( \langle b_j\rangle \langle b_j\rangle  \big)^{1/2}
= \langle |b_j| \rangle,
\end{align*}
and thus
\[
\left\langle \big( |\vect{b}^{(\lambda)}| \big)_\lambda \right\rangle
= \sum_{j=0}^n \left\langle \big(  |b_j^{(\lambda)}| \big)_\lambda \right\rangle
\leq \sum_{j=0}^n \left\langle  |b_j|  \right\rangle
= \left\langle |\vect{b}|  \right\rangle
\]
Since $\langle ( |\vect{b}^{(\lambda)}| )_\lambda \rangle$ and $\langle |\vect{b}| \rangle$ commute, we get
\[
\left\langle \big( (|\vect{b}^{(\lambda)}|-\delta)_+ \big)_\lambda \right\rangle
\leq \left\langle (|\vect{b}|-\delta)_+ \right\rangle.
\]
It follows that $\langle (|\vect{b}^{(\lambda)}|-\delta)_+\rangle (1-h^d)=0$, and thus
\begin{align}
\label{eq:gr0_ideals:Step2b}
\lim_{\lambda } \left\| (|\vect{b}^{(\lambda)}|-\delta)_+ h_\lambda^d - (|\vect{b}^{(\lambda)}|-\delta)_+ \right\| = 0.
\end{align}

Decompose $q$ as $q=\sum_{d=1}^D q_d$, where $q_d$ is homogeneous of degree $d$.
For each $d\in\{1,\ldots,D\}$, we have $\lim_\lambda \| h_\lambda^d q_d(\vect{b}) - q_d(\vect{b}^{(\lambda)}) \| = 0$ by \autoref{prp:gr0_ideals_lemma1}.
Using \eqref{eq:gr0_ideals:Step2b}, it follows that
\[
\lim_\lambda \left\| (|\vect{b}^{(\lambda)}|-\delta)_+ q_d(\vect{b}) 
- (|\vect{b}^{(\lambda)}|-\delta)_+ q_d(\vect{b}^{(\lambda)} \right\| = 0,
\]
from which we deduce \eqref{eq:gr0_ideals:Step2a}.

\emph{Step~3:}
We have $\vect{b}=_\eta\vect{y}$ by \eqref{eq:gr0_ideals:Step1D}, and thus $\|\vect{b}^{(\lambda)}-\vect{y}^{(\lambda)}\| \leq \| \vect{b}-\vect{y} \| < \eta$ for all $\lambda$.
Using that $\lim_\lambda\| \vect{y}^{(\lambda)}-\vect{y}\|=0$, we can choose $\lambda_0\in\Lambda$ such that
\[
\vect{b}^{(\lambda)}=_\eta \vect{y}
\]
for all $\lambda\geq\lambda_0$.
Using \eqref{eq:gr0_ideals:Step2a}, we choose $\lambda\geq\lambda_0$ such that
\begin{align}
\label{eq:gr0_ideals:Step3}
(|\vect{b}^{(\lambda)}|-\delta)_+ q(\vect{b}) 
=_{\varepsilon/3M} (|\vect{b}^{(\lambda)}|-\delta)_+ q(\vect{b}^{(\lambda)}).
\end{align}

Let us check that $\vect{b}^{(\lambda)} \in I^{n+1}_\sa$ has the desired properties.
Since $\lambda\geq\lambda_0$, we have $\vect{b}^{(\lambda)}=_\eta\vect{y}$.
Using that $\eta\leq\tfrac{\varepsilon}{2}$, and using \eqref{eq:gr0_ideals:Step1A}, we get
\[
\vect{b}^{(\lambda)}
=_{\varepsilon/2} \vect{y}
=_{\varepsilon/2} \vect{x}.
\]

Since $\vect{b}^{(\lambda)}=_\eta\vect{y}$, by \eqref{eq:gr0_ideals:Step1C}, we have
\[
w =_{\varepsilon/3} p(\vect{b}^{(\lambda)})\cdot (|\vect{b}^{(\lambda)}|-\delta)_+ w, \quad
\left\| p(\vect{b}^{(\lambda)}) \right\| \leq M, \quad
\left\| (|\vect{b}^{(\lambda)}|-\delta)_+ \right\| \leq M.
\]

Using this at the first step,
using that $w=_{\varepsilon/3M^2}q(\vect{b})$ (see \eqref{eq:gr0_ideals:Step1D}) at the second step, 
and using \eqref{eq:gr0_ideals:Step3} at the third step, we get
\begin{align*}
w 
&=_{\varepsilon/3} p(\vect{b}^{(\lambda)})\cdot (|\vect{b}^{(\lambda)}|-\delta)_+ w \\
&=_{\varepsilon/3} p(\vect{b}^{(\lambda)})\cdot (|\vect{b}^{(\lambda)}|-\delta)_+ q(\vect{b}) \\
&=_{\varepsilon/3} p(\vect{b}^{(\lambda)})\cdot (|\vect{b}^{(\lambda)}|-\delta)_+ q(\vect{b}^{(\lambda)}).
\end{align*}

Since $p(\vect{b}^{(\lambda)})\cdot (|\vect{b}^{(\lambda)}|-\delta)_+ \cdot q(\vect{b}^{(\lambda)})$ belongs to $C^*(\vect{b}^{(\lambda)})$, we have $w\in_\varepsilon C^*(\vect{b}^{(\lambda)})$.
Thus $\vect{z}:=\vect{b}^{(\lambda)}$ has the desired properties.
\end{proof}

The next result is the key tool to construct generators in separable \ca{s}.
The basic idea was obtained together with Strung, Tikuisis, Orovitz and White at the workshop `Set theory and \ca{s}' at the AIM in Palo Alto, January 2012.

\begin{prp}
\label{prp:building_generators}
Let $A$ be a \ca{}, let $n\in\mathbb{N}$ such that $\grPre(A)\leq n$, let $\vect{a}\in A^{n+1}_\sa$, $\varepsilon>0$, and let $(c_j)_{j\in\mathbb{N}}$ be a sequence in $A$.
Then there exists $\vect{b}\in A^{n+1}_\sa$ such that
\[
\vect{b}=_\varepsilon\vect{a}, \quad\text{ and }\quad
c_0,c_1,\ldots \in C^*(\vect{b}).
\]
\end{prp}
\begin{proof}
Let $(d_k)_{k\geq 1}$ be a sequence in $A$ such that for each $j\in\mathbb{N}$ there are infinitely many $k\geq 1$ with $d_k=c_j$.
We may assume that $d_1=0$.
We inductively find tuples $\vect{b}_1,\vect{b}_2,\ldots \in A^{n+1}_\sa$ and numbers $\delta_1,\delta_2,\ldots>0$ such
\begin{enumerate}
\item
$\|\vect{b}_k-\vect{b}_{k-1}\| <\min\{ \tfrac{\delta_1}{2^{k-1}}, \tfrac{\delta_2}{2^{k-2}},\ldots,\tfrac{\delta_{k-1}}{2} \}$, for $k\geq 2$; and
\item
$d_k \in_{1/k} C^*(\vect{b}')$ for all $\vect{b}'\in A^{n+1}_\sa$ with $\| \vect{b}' - \vect{b}_k \| \leq \delta_k$, for $k\geq 1$.
\end{enumerate}

Set $\vect{b}_1:=\vect{a}$ and $\delta_1:=\tfrac{\varepsilon}{2}$.
We have $d_1=0\in C^*(\vect{b}_1)$, which shows that~(2) is satisfied for $k=1$.

Assume $\vect{b}_1,\ldots,\vect{b}_{k-1}$ and $\delta_1,\ldots,\delta_{k-1}$ have been obtained for some $k\geq 2$.
Using that $\grPre(A)\leq n$, we obtain $\vect{b}_k\in A^{n+1}_\sa$ such that
\[
\|\vect{b}_k-\vect{b}_{k-1}\| < \min \big\{ \frac{\delta_1}{2^{k-1}}, \frac{\delta_2}{2^{k-2}},\ldots,\frac{\delta_{k-1}}{2} \big\}, \quad\text{ and }\quad
d_k \in_{1/k} C^*(\vect{b}_k).
\]
Choose a nc-polynomial $p_k$ such that $d_k=_{1/k}p_k(\vect{b}_k)$.
Then $p_k$ is continuous as a function $A^{n+1}\to A$, which allows us to choose $\delta_k>0$ satisfying~(2).

Thus, we obtain $\vect{b}_k\in A^{n+1}_\sa$ and $\delta_k>0$ satisfying~(1) and~(2).
Condition~(1) implies that $(\vect{b}_k)_{k\geq 1}$ is a Cauchy sequence.
Set $\vect{b}:=\lim_{k\to\infty}\vect{b}_k \in A^{n+1}_\sa$, and let us check that it has the desired properties.

For each $k\geq 1$ and $l\geq 1$, repeated application of~(1) implies
\[
\| \vect{b}_{k+l}-\vect{b}_k \|
\leq \sum_{j=1}^{l} \| \vect{b}_{k+j}-\vect{b}_{k+j-1} \|
< \sum_{j=1}^{l} \frac{\delta_k}{2^j} \leq \delta_k,
\]
and thus $\|\vect{b}-\vect{b}_k\| = \lim_{l\to\infty} \| \vect{b}_{k+l}-\vect{b}_k \| \leq \delta_k$.
In particular,
\[
\|\vect{b}-\vect{a}\| 
= \|\vect{b}-\vect{b}_1\| 
\leq \delta_1 
= \frac{\varepsilon}{2}
< \varepsilon.
\]

Further, condition~(2) ensures that $d_k\in_{1/k}C^*(\vect{b})$ for all $k\geq 1$.
Then, for each $j\in\mathbb{N}$, it follows that $c_j\in C^*(\vect{b})$, since $c_j$ was assumed to appear infinitely many times in the sequence $(d_k)_k$.
\end{proof}

\begin{prp}
\label{prp:gr0_extensions}
Let $A$ be a \ca{}, and let $I\subseteq A$ be an ideal.
Then:
\begin{align*}
\grPre(A) \leq \grPre(I)+\grPre(A/I)+1.
\end{align*}
\end{prp}
\begin{proof}
Set $B:=A/I$.
Set $m:=\grPre(I)$ and $n:=\grPre(B)$, which we may assume to be finite.
Let $\pi\colon A\to B$ denote the quotient map.
It induces a natural map $A^{n+1}\to B^{n+1}$, also denote by $\pi$.
To verify $\grPre(A)\leq m+n+1$, let $\vect{x}\in A^{m+1}_\sa$, $\vect{y}\in A^{n+1}_\sa, \varepsilon>0$ and $z\in A$.
We need to find $\vect{x}'\in A^{m+1}_\sa$ and $\vect{y}'\in A^{n+1}_\sa$ such that
\[
\vect{x}'=_\varepsilon\vect{x},\quad
\vect{y}'=_\varepsilon\vect{y}, \quad\text{ and }\quad
z\in_\varepsilon C^*(\vect{x}', \vect{y}').
\]

Set $\vect{b}:=\pi(\vect{y})$.
Let $x_0,\ldots,x_m\in A_\sa$ such that $\vect{x}=(x_0,\ldots,x_m)$.
Since $\grPre(B)\leq n$, we can apply \autoref{prp:building_generators} to obtain $\vect{b}'\in B^{n+1}_\sa$ such that
\[
\vect{b}'=_\varepsilon\vect{b}, \quad\text{ and }\quad
\pi(z), \pi(x_0), \ldots, \pi(x_m) \in C^*(\vect{b}').
\]
Let $\vect{y}'\in A^{n+1}_\sa$ be a lift of $\vect{b}'$ with $\vect{y}'=_\varepsilon\vect{y}$.
Choose $w,a_0,\ldots,a_m\in C^*(\vect{y}')$ such that
\[
\pi(w)=\pi(z),\ \pi(a_0)=\pi(x_0),\ \ldots, \ \pi(a_m)=\pi(x_m).
\]
Set $\vect{a}=(a_0,\ldots,a_m)\in A^{m+1}_\sa$.
Then $\vect{x}-\vect{a}\in I^{m+1}_\sa$ and $z-w\in I$.
Since $\grPre(I)\leq m$, we can apply \autoref{prp:building_generators} to obtain $\vect{c}\in I^{m+1}_\sa$ such that
\[
\vect{c}=_\varepsilon\vect{x}-\vect{a}, \quad\text{ and }\quad z-w\in C^*(\vect{c}).
\]
Set $\vect{x}':=\vect{a}+\vect{c}$.
Then $\vect{x}'$ and $\vect{y}'$ have the desired properties.
\end{proof}

\begin{prp}
\label{prp:gr0_sepSub}
Let $A$ be a \ca{} and let $B_0\subseteq A$ be a separable sub-\ca.
Then there exists a separable sub-\ca{} $B\subseteq A$ such that $B_0\subseteq B$ and $\grPre(B)\leq\grPre(A)$.
\end{prp}
\begin{proof}
Set $n:=\grPre(A)$, which we may assume to be finite.
We will inductively obtain:
\begin{itemize}
\item
separable sub-\ca{s} $B_0\subseteq B_1\subseteq \ldots A$;
\item
a dense sequence $(\vect{b}^{(j,k)})_{k\in\mathbb{N}}$ in $(B_j)^{n+1}_\sa$, for each $j\in\mathbb{N}$;
\item
a sequence $(\vect{a}^{(j,k)})_{k\in\mathbb{N}}$ in $A^{n+1}_\sa$ such that
\[
\vect{a}^{(j,k)}=_{1/(j+1)}\vect{b}^{(j,k)}, \quad\text{ and }\quad
B_j\subseteq C^*(\vect{a}^{(j,k)})
\]
for each $k,j\in\mathbb{N}$.
\end{itemize} 
The algebra $B_0$ is given.
Assuming that we have obtained $B_j$, for some $j\in\mathbb{N}$, we choose any sequence $(\vect{b}^{(j,k)})_k$ that is dense in $(B_j)^{n+1}_\sa$.
Then, for each $k\in\mathbb{N}$, using that $\grPre(A)\leq n$, we apply \autoref{prp:building_generators}
for $\vect{b}^{(j,k)}\in A^{n+1}_\sa$, for $1/(j+1)$ and for any sequence in $B_j$ that generates $B_j$ as a \ca{} to obtain $\vect{a}^{(j,k)}\in A^{n+1}_\sa$ such that $\vect{a}^{(j,k)}=_{1/(j+1)}\vect{b}^{(j,k)}$ and $B_j\subseteq C^*(\vect{a}^{(j,k)})$.
Then set
\[
B_{j+1}:= C^*(A_j,\vect{a}^{(j,0)},\vect{a}^{(j,1)},\ldots) \subseteq A.
\]
Then $B_{j+1}$ is a separable sub-\ca{} of $A$.

Set $B:=\overline{\bigcup_j B_j}\subseteq A$, which is a separable sub-\ca{} of $A$ such that $B_0\subseteq B$.
To verify $\grPre(B)\leq n$, let $\vect{x}\in B^{n+1}_\sa, \varepsilon>0$ and $z\in B$.
We need to find $\vect{y}\in B^{n+1}_\sa$ such that
\[
\vect{y}=_\varepsilon\vect{x}, \quad\text{ and }\quad
z\in_\varepsilon C^*(\vect{y}).
\]
Using the definition of $B$, choose $j\in\mathbb{N}$ such that
\[
z,x_0,\ldots,x_n\in_{\varepsilon/2} B_j.
\]
We may assume that $j$ is so large that $\tfrac{1}{j+1}<\tfrac{\varepsilon}{2}$.
Since $(\vect{b}^{(j,k)})_k$ is dense in $(B_j)^{n+1}_\sa$, we obtain $k\in\mathbb{N}$ such that
\[
\vect{x}=_{\varepsilon/2} \vect{b}^{(j,k)}.
\]
Now set $\vect{y}:=\vect{a}^{(j,k)} \in (B_{j+1})^{n+1}_\sa \subseteq B^{n+1}_\sa$.
Then
\[
\| \vect{y} - \vect{x} \|
= \| \vect{a}^{(j,k)} - \vect{x} \|
\leq \| \vect{a}^{(j,k)} - \vect{b}^{(j,k)} \| + \| \vect{b}^{(j,k)} - \vect{x} \|
< \frac{1}{j+1} + \frac{\varepsilon}{2} \leq \varepsilon.
\]
Further, we have
\[
z\in_\varepsilon B_j \subseteq C^*(\vect{a}^{(j,k)}) = C^*(\vect{y}). \qedhere
\]
\end{proof}

\begin{pgr}
\label{pgr:ncDimThy}
In \cite[Definition~1]{Thi13TopDimTypeI}, I formalized the concept of a noncommutative dimension theory by proposing a set of axioms that such a theory should satisfy.
These axioms are generalizations of properties of the (local) dimension of locally compact, Hausdorff spaces, and it was shown that they are satisfied by many theories, in particular the real and stable rank, the topological dimension, the decomposition rank and the nuclear dimension.

A \emph{dimension theory} (defined on the class of all \ca{s}) is an assignment that to each \ca{} $A$ associates  $d(A)\in\{0,1,2,\ldots,\infty\}$ such that:
\begin{enumerate}
\item[\axiomD{1}]
$d(I)\leq d(A)$ whenever $I\subseteq A$ is an ideal in a \ca{} $A$;
\item[\axiomD{2}]
$d(A/I)\leq d(A)$ whenever $I\subseteq A$ is an ideal in a \ca{} $A$;
\item[\axiomD{3}]
$d(A\oplus B)=\max\{d(A),d(B)\}$, whenever $A$ and $B$ are \ca{s};
\item[\axiomD{4}]
$d(\widetilde{A})=d(A)$ for every \ca{} $A$;
\item[\axiomD{5}]
if a \ca{} $A$ is approximated by a family of sub-\ca{s} $A_\lambda\subseteq A$, and if $n\in\mathbb{N}$ is such that $d(A_\lambda)\leq n$ for all $\lambda$, then  $d(A)\leq n$;
\item[\axiomD{6}]
given a \ca{} $A$ and a separable sub-\ca{} $A_0\subseteq A$, there exists a separable sub-\ca{} $B\subseteq A$ such that $A_0\subseteq B$ and $d(B)\leq d(A)$.
\end{enumerate}

We have shown that $\grPre$ satisfies \axiomD{1}, \axiomD{2}, \axiomD{5} and \axiomD{6}.
It is unclear if $\grPre$ satisfies \axiomD{4} (see \autoref{qst:gr_grPre}), which is the reason for defining the generator rank of $A$ as $\grPre(\widetilde{A})$.
In \autoref{sec:permanence}, we will show that the generator rank satisfies \axiomD{1}, \axiomD{2}, and \axiomD{4}-\axiomD{6}.
It is unclear if $\grPre$ satisfies \axiomD{3}:
\end{pgr}

\begin{qst}
\label{qst:gr0_sums}
Do we have $\grPre(A\oplus B)=\max\{\grPre(A),\grPre(B)\}$ for all $A$ and $B$?
\end{qst}

\section{The generator rank}
\label{sec:gr}

We define the generator rank of a \ca{} $A$ as the invariant $\grPre$ developed in \autoref{sec:grPre} applied to the minimal unitization $\widetilde{A}$;
see \autoref{dfn:gr}.
It follows that $\grPre(A)$ and $\gr(A)$ are closely related, and in many cases we know that they agree;
see \autoref{prp:grPre_less_gr} and \autoref{pgr:gr_grPre}.

The generator rank of a \emph{separable}, unital \ca{} $A$ is the smallest integer $n$ such that the self-adjoint $(n+1)$-tuples in $A$ that generate $A$ as a \ca{} are dense in $A^{n+1}_\sa$;
see \autoref{prp:grSep}.
This is similar to the definition of the real rank, which also explains the index shift;
see also \autoref{rmk:gr_connection_rr}.

\begin{dfn}
\label{dfn:gr}
Let $A$ be a \ca{}.
The \emph{generator rank} of $A$ is defined as $\gr(A):=\grPre(\widetilde{A})$.
Thus, $\gr(A)$ is the smallest integer $n\geq 0$ such that for every $\vect{a}\in (\widetilde{A})^{n+1}_\sa$, $\varepsilon>0$ and $c\in A$, there exists $\vect{b}\in (\widetilde{A})^{n+1}_\sa$ such that
\[
\vect{b}=_\varepsilon\vect{a}, \quad\text{ and }\quad
c\in_\varepsilon C^*(\vect{b}).
\] 
If no such $n$ exists, we have $\gr(A)=\infty$.
\end{dfn}

\begin{ntn}
\label{pgr:notation_gen}
Let $A$ be a \ca{}. 
We denote the set of generating (self-adjoint) $n$-tuples by:
\begin{align*}
\Gen_n(A) := \big\{ \vect{a}\in A^n : A=C^*(\vect{a}) \big\}, \quad
\Gen_n(A)_\sa :=\Gen_n(A)\cap A^n_\sa.
\end{align*}
\end{ntn}

\begin{lma}
\label{prp:Gen-G-delta}
Let $A$ be a \ca{} and $n\geq 1$.
Then $\Gen_n(A)\subseteq A^n$ and $\Gen_n(A)_\sa\subseteq A^n_\sa$ are $G_\delta$-subsets.
\end{lma}
\begin{proof}
It suffices to show that $\Gen_n(A)\subseteq A^n$ is a $G_\delta$-subset. 
Since the empty set is $G_\delta$, we may assume that $\Gen_k(A)\neq\emptyset$, which in turn implies that $A$ is separable.
Let $c_1,c_2,\ldots$ be a dense sequence in $A$.
For $k\geq 1$, define:
\begin{align*}
U_k := \big\{ \vect{a}\in A^n : c_1,\ldots,c_k\in_{1/k}C^*(\vect{a}) \big\}.
\end{align*}
Since $\Gen_n(A)=\bigcap_{k\geq 1} U_k$, it suffices to prove that each $U_k$ is open.

Let $k\geq 1$ and $\vect{a}\in U_k$.
Then there exist nc-polynomials $p_1,\ldots,p_k$ such that $c_j=_{1/k}p_j(\vect{a})$ for $j=1,\ldots,k$.
We may consider each $p_j$ as a function $A^n\to A$, which is clearly continuous.
Therefore, for each $j$, there exists $\delta_j>0$ such that $c_j=_{1/k}p_j(\vect{b})$ for all $\vect{b}\in A^n$ with $\vect{b}=_{\delta_j}\vect{a}$.
Then the open ball around $\vect{a}$ with radius $\min\{\delta_1,\ldots,\delta_k\}$ is contained in $U_k$.
This proves that $U_k$ is open.
\end{proof}

The next result is a direct consequence of \autoref{prp:building_generators}.

\begin{thm}
\label{prp:grSep}
Let $A$ be a separable \ca{} and $n\in\mathbb{N}$.
Then:
\begin{enumerate}
\item
$\grPre(A)\leq n$ if and only if $\Gen_{n+1}(A)_\sa\subseteq A^{n+1}_\sa$ is a dense $G_\delta$-subset.
\item
$\gr(A)\leq n$ if and only if $\Gen_{n+1}(\widetilde{A})_\sa\subseteq (\widetilde{A})^{n+1}_\sa$ is a dense $G_\delta$-subset.
\end{enumerate} 
\end{thm}
\begin{proof}
By definition of $\gr(A)$, it suffices to prove the first statement.
Assume that $\grPre(A)\leq n$.
By \autoref{prp:Gen-G-delta}, $\Gen_{n+1}(A)_\sa\subseteq A^{n+1}_\sa$ is a $G_\delta$-subset.
To show that it is dense, let $\vect{a}\in A^{n+1}_\sa$ and $\varepsilon>0$.
Since $A$ is separable, there exists a sequence $(c_j)_{j\in\mathbb{N}}$ in $A$ such that $A=C^*(c_0,c_1,\ldots)$.
Applying \autoref{prp:building_generators}, we obtain $\vect{b}\in A^{n+1}_\sa$ such that
\[
\vect{b}=_\varepsilon\vect{a}, \quad\text{ and }\quad
c_0,c_1,\ldots \in C^*(\vect{b}).
\]
Then $A=C^*(\vect{b})$, which means that $\vect{b}$ belongs to $\Gen_{n+1}(A)_\sa$.

The converse implication is obvious.
\end{proof}

\begin{lma}
\label{prp:Gen-sa-nonsa}
Let $A$ be a \ca{} and $n\geq 1$.
Then $\Gen_n(A)\subseteq A^n$ is dense if and only if $\Gen_{2n}(A)_\sa\subseteq A^{2n}_\sa$ is dense.
\end{lma}
\begin{proof}
Let $\Phi\colon A^{2n}_\sa\to A^n$ be given by
\[
\Phi(a_1,\ldots,a_{2n}) := (a_1+ia_{n+1},\ldots,a_n+ia_{2n})\in A^n
\]
for $(a_1,\ldots,a_{2n}) \in A^{2n}_\sa$.

In general, two self-adjoint elements $c,d\in A_\sa$ generate the same sub-\ca{} as the element $c+id$.
It follows that $C^*(\vect{a})=C^*(\Phi(\vect{a}))\subseteq A$ for every $\vect{a}\in A^{2n}_\sa$.
Hence, $\Phi$ maps $\Gen_{2n}(A)_\sa$ onto $\Gen_n(A)$.
Since $\Phi$ is a homeomorphism, we deduce that $\Gen_{2n}(A)_\sa\subseteq A^{2n}_\sa$ is dense if and only if $\Gen_{n}(A)\subseteq A^{n}$ is.
\end{proof}

As a direct consequence of \autoref{prp:grSep} and \autoref{prp:Gen-sa-nonsa}, we have:

\begin{prp}
\label{prp:connection_gr_sa-nonsa}
Let $A$ be a separable \ca{} and $n\geq 1$.
Then:
\begin{enumerate}
\item
$\grPre(A)\leq 2n-1$ if and only if $\Gen_n(A)\subseteq A^n$ is a dense $G_\delta$-subset.
\item
$\gr(A)\leq 2n-1$ if and only if $\Gen_n(\widetilde{A})\subseteq (\widetilde{A})^n$ is a dense $G_\delta$-subset.
\end{enumerate} 
\end{prp}

\begin{rmk}
\label{rmk:gr1_means_generic}
In \autoref{prp:charGr0}, we show that a \ca{} $A$ satisfies $\gr(A)=0$ if and only if $A$ is commutative with totally disconnected spectrum.
Thus, for noncommutative \ca{s}, the lowest possible (and therefore most interesting) value of the generator rank is one.

Let $A$ be a unital \ca{}.
Using the homeomorphism between $A$ and $A^2_\sa$ (as in the proof of \autoref{prp:Gen-sa-nonsa}), we see that $\gr(A)\leq 1$ if and only if for every $a,c\in A$ and $\varepsilon>0$ there exists $b\in A$ such that
\[
\|b-a\|<\varepsilon, \quad\text{ and }\quad  c \in C^*(b).
\]
If $A$ is unital and separable, then $\gr(A)\leq 1$ if and only if a generic element in $A$ is a generator;
see \autoref{prp:connection_gr_sa-nonsa},

If $A$ is nonunital and separable, then $\gr(A)\leq 1$ implies $\grPre(A)\leq 1$ (by \autoref{prp:grPre_less_gr}) and thus a generic element in $A$ is a generator.
However, it is unclear if a nonunital \ca{} with a dense set of generators has generator rank at most one;
see \autoref{pgr:gr_grPre}.
\end{rmk}

\begin{pgr}
\label{pgr:rr_sr}
The generator rank is connected to the stable rank introduced by Rieffel \cite{Rie83DimSRKThy} and the real rank of Brown and Pedersen \cite{BroPed91CAlgRR0}.
To recall the definitions, let $A$ be a unital \ca{} and $n\geq 1$.
Set
\begin{align*}
\Lg_n(A) := \left\{ \vect{a}\in A^n : \sum_{j=1}^n a_j^\ast a_j \text{ invertible} \right\}, \quad\text{ and }\quad
\Lg_n(A)_\sa := \Lg_n(A)\cap A^n_\sa.
\end{align*}

The abbreviation `Lg' stands for `left generators', and the reason is that a tuple $\vect{a}\in A^n$ belongs to $\Lg_n(A)$ if and only if the elements $a_1,\ldots,a_n$ generate $A$ as a (not necessarily closed) left ideal, that is, $Aa_1+\ldots+Aa_n=A$.

The (topological) stable rank of $A$, denoted $\sr(A)$, is the smallest integer $n\geq 1$ such that $\Lg_n(A)$ is dense in $A^n$;
see \cite[Definition~1.4]{Rie83DimSRKThy}.
The real rank of $A$, denoted $\rr(A)$, is the smallest integer $n\geq 0$ such that $\Lg_{n+1}(A)_\sa$ is dense in $A^{n+1}_\sa$;
see \cite{BroPed91CAlgRR0}.
If $A$ is nonunital, one sets $\sr(A):=\sr(\widetilde{A})$ and $\rr(A):=\rr(\widetilde{A})$.

Notice the index shift in the definition of the real rank (as opposed to the definition of stable rank).
It leads to nicer formulas, for example $\rr(C(X))=\dim(X)$ if $X$ is a compact, Hausdorff space.
We use the same index shift in Definitions~\ref{dfn:grPre} and~\ref{dfn:gr} since the generator rank is more closely connected to the real rank than to the stable rank as explained in \autoref{rmk:gr_connection_rr}.
\end{pgr}

\begin{lma}
\label{prp:charLg}
Let $A$ be a unital \ca{}, $n\geq 1$, and $\vect{a}\in A^n_\sa$.
Then $\vect{a}$ belongs to $\Lg_n(A)_\sa$ if and only if $1\in C^*(\vect{a})$.
\end{lma}
\begin{proof}
Let us show the forward implication.
If $\vect{a}\in\Lg_n(A)$, then $a:=\sum_{j=1}^n a_j^\ast a_j$ is invertible.
Since $a\in C^*(\vect{a})$, it follows that $1\in C^*(\vect{a})$.

To show the backwards implication, assume that $1\in C^*(\vect{a})$.
Choose a nc-polynomial $p$ such that $\|1-p(\vect{a})\|<1$.
Then $p(\vect{a})$ is invertible, and we denote its inverse by $v\in A$.
Write $p$ as a sum of polynomials, $p=\sum_{j=1}^kp_j$, where each $p_j$ is of the form $p_j(\vect{a})=q_j(\vect{a})\cdot a_j$ for some other nc-polynomial $q_j$.
Then:
\[
1 = v \cdot p(\vect{a}) = \sum_{j=1}^k \left( v\cdot q_j(\vect{a}) \right) a_j,
\]
which shows that $\vect{a}$ generates $A$ as a left ideal.
\end{proof}

\begin{prp}
\label{prp:rr_less_gr}
Let $A$ be a \ca{}.
Then $\rr(A) \leq \gr(A)$.
\end{prp}
\begin{proof}
By passing to the minimal unitization, if necessary, we may assume that $A$ is unital.
Set $n:=\gr(A)$, which we may assume to be finite.
To verify that $\rr(A)\leq n$, let $\vect{a}\in A^{n+1}_\sa$ and $\varepsilon>0$.
We need to find $\vect{b}\in\Lg_{n+1}(A)_\sa$ such that $\vect{b}=_\varepsilon\vect{a}$.
We may assume that $\varepsilon<1$.

Using that $\grPre(A)=\gr(A)\leq n$, we obtain $\vect{b}\in A^{n+1}_\sa$ such that
\[
\vect{b}=_\varepsilon\vect{a}, \quad\text{ and }\quad
1\in_\varepsilon C^*(\vect{b}).
\]
Since $\varepsilon<1$, it follows that $C^*(\vect{b})$ contains an invertible element, and thus $1\in C^*(\vect{b})$.
By \autoref{prp:charLg}, we obtain $\vect{b}\in\Lg_{n+1}(A)_\sa$, as desired.
\end{proof}

\begin{rmk}
\label{rmk:gr_connection_rr}
Let us explain why the generator rank is more closely related to the real rank than to the stable rank.
Let $A$ be a unital, separable \ca{}.
It follows from \autoref{prp:charLg} that $\Gen_n(A)_\sa\subseteq\Lg_n(A)_\sa$ for every $n\geq 1$, which directly implies that $\rr(A)\leq\gr(A)$.
(By \autoref{prp:rr_less_gr}, this inequality holds in general.)

The analog for non-selfadjoint (left) generators does not hold.
Consider $B:=M_3(C([0,1]^2))$.
Then $\sr(B)=2$ and so $\Lg_1(B)$ is not dense in $B$.
On the other hand, we have $\gr(B)=1$ by Theorem~4.17 and Remark~4.20 in \cite{Thi20arX:grSubhom}, and so $\Gen_1(B)$ is dense in $B$ by \autoref{prp:connection_gr_sa-nonsa}.
Thus, $\Gen_1(B) \nsubseteq \Lg_1(B)$, and there exists $b\in B$ such that $1\in C^*(b)$ but $b\notin\Lg_1(B)$.

Next, let us consider the variant of the generator rank, defined as the smallest integer $n\geq 1$ such that $\Gen_n(A)$ is dense in $A^n$.
Let us denote this value by $\gr'(A)$.
Given that $\gr$ is related to the real rank, one might expect that $\gr'$ is related to the stable rank.
This is, however, not the case either.
For instance, unlike the real and stable rank, the invariants $\gr$ and $\gr'$ are very closely tied together.
Indeed, it follows from \autoref{prp:connection_gr_sa-nonsa} that
\begin{align*}
\gr'(A) = \left\lceil \frac{\gr(A)+1}{2} \right\rceil.
\end{align*}
Moreover, while the estimate $\rr(A)\leq\gr(A)$ always holds (see \autoref{prp:rr_less_gr}), the above example $B:=M_3(C([0,1]^2))$ satisfies $\sr(B)=2\nleq 1=\gr'(B)$.
On the other hand, we have $\sr(C([0,1]))=1 \ngeq 2 = \gr'(C([0,1]))$.
Thus, we neither have $\sr\leq\gr'$ nor $\sr\geq\gr'$.
\end{rmk}

\begin{prp}
\label{prp:grPre_less_gr}
Let $A$ be a \ca{}.
Then $\grPre(A)\leq\gr(A)\leq\grPre(A)+1$.
\end{prp}
\begin{proof}
The statement is clear if $A$ is unital.
So assume that $A$ is nonunital.
Then $A$ is an ideal in $\widetilde{A}$ such that $\widetilde{A}/A\cong\mathbb{C}$.
By \autoref{prp:gr0_ideals}, we have
\[
\grPre(A)\leq\grPre(\widetilde{A})=\gr(A).
\]
Further, we have $\grPre(\mathbb{C})=0$, and therefore
\[
\gr(A)=\grPre(\widetilde{A})\leq\grPre(A)+\grPre(\mathbb{C})+1=\grPre(A)+1,
\]
by \autoref{prp:gr0_extensions}.
\end{proof}

\begin{thm}
\label{prp:gr_grPreMaxRR}
Let $A$ be a \ca{}.
Then $\gr(A) =\max\{\rr(A), \grPre(A)\}$.
\end{thm}
\begin{proof}
By Propositions~\ref{prp:rr_less_gr} and~\ref{prp:grPre_less_gr}, we have $\gr(A)\geq\max\{\rr(A),\grPre(A)\}$.
The converse inequality is clear if $A$ is unital.

So assume that $A$ is nonunital, and set $n:=\max\{\rr(A), \grPre(A)\}$, which we may assume to be finite.
We let $1\in\widetilde{A}$ denote the adjoint unit.
To verify $\gr(A)\leq n$, let $\vect{x}\in(\widetilde{A})^{n+1}_\sa, \varepsilon>0$ and $z\in\widetilde{A}$.
We need to find $\vect{y}\in(\widetilde{A})^{n+1}_\sa$ such that
\[
\vect{y}=_\varepsilon\vect{x},\quad\text{ and }\quad
z\in_\varepsilon C^*(\vect{y}).
\]

Since $\rr(A)\leq n$, there is $\vect{x}'\in\Lg_{n+1}(\widetilde{A})_\sa$ with $\vect{x}'=_{\varepsilon/2}\vect{x}$.
Since $\Lg_{n+1}(\widetilde{A})_\sa$ is an open subset of $\widetilde{A}^{n+1}_\sa$, there exists $\delta>0$ such that every $\vect{y}\in(\widetilde{A})^{n+1}_\sa$ with $\vect{y}=_\delta\vect{x}'$ belongs to $\Lg_{n+1}(\widetilde{A})_\sa$.
We may assume that $\delta<\varepsilon/2$.

Let $\vect{a}'\in A^{n+1}_\sa$, $\vect{r} \in (\mathbb{R} 1)^{n+1}\subseteq (\widetilde{A})^{n+1}_\sa$, $c\in A$ and $t\in\mathbb{C} 1 \subseteq\widetilde{A}$ be the unique elements such that
\[
\vect{x}' = \vect{a}' + \vect{r}, \quad\text{ and }\quad
z = c + t.
\]
Since $\grPre(A)\leq n$, we can apply \autoref{prp:building_generators} to obtain $\vect{b}\in A^{n+1}_\sa$ such that
\[
\vect{b}=_{\delta}\vect{a}', \quad\text{ and }\quad
c \in C^*(\vect{b}).
\]

Set $\vect{y}:=\vect{b}+\vect{r}$.
Let us show that $\vect{y}$ has the desired properties.

We have $\vect{y}=_\delta\vect{x}'$, and therefore $\vect{y}=_{\varepsilon/2}\vect{x}'=_{\varepsilon/2}\vect{x}$.
Further, by choice of $\delta$, we have $\vect{y}\in\Lg_{n+1}(\widetilde{A})_\sa$.
It follows that $1\in C^*(\vect{y})$, and so $b_j=y_j-r_j 1\in C^*(\vect{y})$ for $j=0,\ldots,n$.
Therefore $c\in C^*(\vect{b})\subseteq C^*(\vect{y})$.
Thus, $z=c+t\in C^*(\vect{y})$.
\end{proof}

\begin{cor}
\label{prp:grEqGrPre_rr1}
Let $A$ be a \ca{} satisfying $\rr(A)\leq 1$.
(For example, this is the case if $\rr(A)=0$, $\sr(A)=1$ or $\gr(A)\leq 1$.)
Then $\grPre(A)=\gr(A)$.
\end{cor}
\begin{proof}
If $\grPre(A)=0$ then $\gr(A)=0$ by \autoref{prp:charGr0}.
If $\grPre(A)\geq 1$, then \autoref{prp:gr_grPreMaxRR} implies
\[
\gr(A) = \max\{\rr(A), \grPre(A)\} = \grPre(A). \qedhere
\]
\end{proof}

\begin{pgr}
\label{pgr:gr_grPre}
Let $A$ be a \ca.
We have $\gr(A)=\grPre(A)$ if $\rr(A)\leq 1$ (\autoref{prp:grEqGrPre_rr1}) or if $A$ is commutative (\autoref{prp:gr_commutative}).
In \cite{Thi20arX:grSubhom}, we show that $\gr(A)=\grPre(A)$ also holds whenever $A$ is subhomogenoeus.
By \autoref{prp:grPre_less_gr} and \autoref{prp:gr_grPreMaxRR}, the only possibility for $\grPre(A)<\gr(A)$ is that
\[
\grPre(A)=n-1, \quad\text{ and }\quad 
\gr(A)=\rr(A)=n
\]
for some $n\geq 2$.
\end{pgr}

\begin{qst}
\label{qst:gr_grPre}
Does every \ca{} $A$ satisfy $\grPre(A)=\gr(A)$?
\end{qst}

\section{Reduction to the separable case}
\label{sec:reduction}

In this short section, we recall some notions from model theory of \ca{s} that will be used in \autoref{sec:commutative} to reduce some proofs to the case of separable \ca{s}.
For details we refer to \cite{FarKat10NonsepUHF1} and \cite{FarHarLupRobTikVigWin06arX:ModelThy}.

\begin{pgr}
\label{pgr:LS}
Given a \ca{} $A$, we let $\SubSep(A)$ denote the collection of separable sub-\ca{s} of $A$.
A family $\mathcal{S}\subseteq\SubSep(A)$ is \emph{$\sigma$-complete} if for every countable, directed subfamily $\mathcal{T}\subseteq\mathcal{S}$ we have $\overline{\bigcup\{B:B\in\mathcal{T}\}}\in\mathcal{S}$.
Further, $\mathcal{S}$ is \emph{cofinal} if for every $B_0\in\SubSep(A)$ there exists $B\in\mathcal{S}$ such that $B_0\subseteq B$.

We say that a property $\mathcal{P}$ of \ca{s} satisfies the \emph{L\"{o}wenheim-Skolem condition} if for every \ca{} $A$ satisfying $\mathcal{P}$, there exists $\sigma$-complete, cofinal subset $\mathcal{S}\subseteq\SubSep(A)$ such that every $B\in\mathcal{S}$ satisfies $\mathcal{P}$.

It is well-known that the intersection of countably many $\sigma$-complete, cofinal subfamilies is again $\sigma$-complete and cofinal.
It follows that the conjunction of countably many properties satisfying the L\"{o}wenheim-Skolem condition is again a property satisfying the L\"{o}wenheim-Skolem condition.

Let $d$ be an assignment from \ca{s} to $\{0,1,2,\ldots,\infty\}$ satisfying \axiomD{5} and \axiomD{6} from \autoref{pgr:ncDimThy}.
Then for each $n\in\mathbb{N}$, the property `$d(\freeVar)\leq n$' satisfies the L\"{o}wenheim-Skolem condition:
If a \ca{} $A$ satisfies $d(A)\leq n$, then the collection $\mathcal{S}\subseteq\SubSep(A)$ of separable sub-\ca{s} $B\subseteq A$ with $d(B)\leq n$ is $\sigma$-complete by \axiomD{5} and cofinal by \axiomD{6}.

Thus, for each $n\in\mathbb{N}$, the properties `$\rr(\freeVar)\leq n$' and `$\grPre(\freeVar)\leq n$' satisfy the L\"{o}wenheim-Skolem condition.
Given a unital \ca{} $A$, it is easy to see that the collection of separable sub-\ca{s} containing the unit of $A$ is $\sigma$-complete and cofinal.
\end{pgr}

Given a \ca{} $A$ and $k\geq 1$, we use $A^{\otimes k}$ to denote the $k$-fold (minimal) tensor product $A\otimes\ldots_k\otimes A$.
If $B\subseteq A$ is a sub-\ca{}, then $B^{\otimes k}$ naturally is a sub-\ca{} of $A^{\otimes k}$;
see \cite[II.9.6.2, p.187]{Bla06OpAlgs}.

\begin{lma}
\label{prp:LSTens}
Let $A$ be a \ca, let $k\geq 1$, and let $\mathcal{S}\subseteq\SubSep(A^{\otimes k})$ be $\sigma$-complete and cofinal.
Then $\{B\in\SubSep(A):B^{\otimes k}\in\mathcal{S}\}$ is $\sigma$-complete and cofinal.
\end{lma}
\begin{proof}
Set $\mathcal{T} := \{B\in\SubSep(A):B^{\otimes k}\in\mathcal{S}\}$.
It is easy to see that $\mathcal{T}$ is $\sigma$-complete.
To show that it is cofinal, let $B_0\in\SubSep(A)$.
We will inductively find increasing sequences $(B_n)_n$ in $\SubSep(A)$ and $(D_n)_n$ in $\mathcal{S}$ such that
\[
B_0^{\otimes k} \subseteq D_0 \subseteq B_1^{\otimes k} \subseteq D_1 \subseteq \ldots.
\]
Assume that we have obtained $B_n$ for some $n\in\mathbb{N}$.
Then $B_n^{\otimes k}\in\SubSep(A^{\otimes k})$.
Since $\mathcal{S}$ is cofinal, we obtain $D_n\in\mathcal{S}$ such that $B_n^{\otimes k}\subseteq D_n$.

For each $d\in D_n$ and $\varepsilon>0$, we can find $m\geq 1$ and $a_j^{(1)},\ldots,a_j^{(k)}\in A$ for $j=1,\ldots,m$ such that
\[
\left\| d - \sum_{j=1}^m a_j^{(1)}\otimes\ldots\otimes a_j^{(k)} \right\| < \varepsilon.
\]
Using this, we find a separable sub-\ca{} $B_{n+1}\subseteq A$ such that $D_n\subseteq B_{n+1}^{\otimes k}$.

Having chosen the sequences $(B_n)_n$ and $(D_n)_n$, set $B:=\overline{\bigcup_n B_n}$.
Then $B$ is a separable sub-\ca{} of $A$ containing $B_0$.
Moreover, by construction, we have
\[
B^{\otimes k} = \overline{ \bigcup_{n=1}^\infty D_n }.
\]
Since $\mathcal{S}$ is $\sigma$-complete, $B^{\otimes k}$ belongs to $\mathcal{S}$.
Thus, $B$ belongs to $\mathcal{T}$, as desired.
\end{proof}

\begin{prp}
\label{prp:rrTensLS}
Let $n\in\mathbb{N}$ and $k\geq 1$.
Then the property `$\rr(\freeVar^{\otimes k})\leq n$' satisfies the L\"{o}wenheim-Skolem condition.
\end{prp}
\begin{proof}
Let $A$ be a \ca{} satisfying $\rr(A^{\otimes k})\leq n$.
Set 
\[
\mathcal{S} := \big\{ B \subseteq A^{\otimes k} : B \text{ separable}, \rr(B)\leq n \big\}.
\]
Then $\mathcal{S}$ is $\sigma$-complete and cofinal, since the real rank is a noncommutative dimension theory (in the sense of \autoref{pgr:ncDimThy}).
Set
\[
\mathcal{T} := \{B\in\SubSep(A) : B^{\otimes k}\in\mathcal{S} \}.
\]
Note that $\mathcal{T}$ consists precisely of the separable sub-\ca{s} $B\subseteq A$ satisfying `$\rr(\freeVar^{\otimes k})\leq n$'.
By \autoref{prp:LSTens}, $\mathcal{T}$ is $\sigma$-complete and cofinal, as desired.
\end{proof}

\section{Generator rank of commutative \texorpdfstring{$C^*$-algebras}{C*-algebras}}
\label{sec:commutative}

In this section, we compute the generator rank of commutative \ca{s};
see \autoref{prp:gr_commutative}.
We first consider commutative \ca{s} that are unital and separable.
We then generalize to the unital, nonseparable case, and finally to the case of arbitrary commutative \ca{s}.

For topological spaces $X$ and $Y$ we denote by $E(X,Y)$ the space of injective, continuous maps $X\to Y$.

Let $X$ be a compact, metric space and $n\in\mathbb{N}$.
We seek to characterize when $\gr(C(X))\leq n$.
Given $\vect{a}\in C(X)^{n+1}_\sa$ and $x\in X$, set $\vect{a}(x):=(a_0(x),\ldots,a_n(x))$.
Since $C(X)$ is unital and separable, by \autoref{prp:grSep}, we have $\gr(C(X))\leq n$ if and only if the elements $\vect{a}\in C(X)^{n+1}_\sa$ that generate $C(X)$ are dense in $C(X)^{n+1}_\sa$.
By the Stone-Weierstrass theorem, $\vect{a}\in C(X)^{n+1}_\sa$ generates $C(X)$ if and only if $\vect{a}(x)\neq 0$ for all $x\in X$  and if $\vect{a}$ separates the points of $X$, that is, $\vect{a}(x)\neq\vect{a}(y)$ whenever $x\neq y$.

After identifying $C(X)^{n+1}_\sa$ with $C(X,\mathbb{R}^{n+1})$, we obtain that $\gr(C(X))\leq n$ if and only if $E(X,\mathbb{R}^{n+1}\setminus\{0\})$ is dense in $C(X,\mathbb{R}^{n+1})$.
This naturally leads to the following questions:
\begin{enumerate}[(1)  ]
\item
When is $C(X,\mathbb{R}^{n+1}\setminus\{0\})\subseteq C(X,\mathbb{R}^{n+1})$ dense?
\item
When is $E(X,\mathbb{R}^{n+1}\setminus\{0\})\subseteq C(X,\mathbb{R}^{n+1}\setminus\{0\})$ dense?
\end{enumerate}

In the proof of \autoref{prp:gr_commutative_sep} we will see that these questions are answered by the (covering) dimension of $X$ and $X\times X$, respectively.

The following result is closely related to the topological concepts of stable and unstable intersection, for which we refer to \cite{Lev18UnstableIntersectConj}.
We note that the result is probably known to the experts, but we could not locate it in the literature in this form.

\begin{prp}
\label{prp:approxByEmb}
Let $X$ be a compact, metric space, and $n\geq 1$.
Then $E(X,\mathbb{R}^n)\subseteq C(X,\mathbb{R}^n)$ is dense if and only if $\dim(X\times X)\leq n-1$.
\end{prp}
\begin{proof}
The backward implication follows from Theorem~1.1 and the Remark below it in \cite{DraRepSce91IntersectCpct}.
To show the forward implication, assume that $E(X,\mathbb{R}^n)\subseteq C(X,\mathbb{R}^n)$ is dense.

\emph{Claim: Let $Y,Z\subseteq X$ be closed, disjoint subsets.
Then $\dim(Y\times Z)\leq n-1$.}
To verify the claim, we first show that $Y$ and $Z$ have \emph{unstable intersection} in $\mathbb{R}^n$, that is, given continuous maps $f\colon Y\to\mathbb{R}^n$ and $g\colon Z\to\mathbb{R}^n$ and $\varepsilon>0$ there exist continuous maps $f'\colon Y\to\mathbb{R}^n$ and $g'\colon Z\to\mathbb{R}^n$ such that $\|f-f'\|_\infty<\varepsilon$, $\|g-g'\|_\infty<\varepsilon$ and $f'(Y)\cap g'(Z)=\emptyset$.
Given such $f, g$ and $\varepsilon>0$, using that $Y$ and $Z$ are disjoint we choose $h\in C(X,\mathbb{R}^n)$ such that $h|_{Y}=f$ and $h|_{Z}=g$.
By assumption, there exists $h'\in E(X,\mathbb{R}^n)$ with $\|h-h'\|_\infty<\varepsilon$.
Then $f':=h'|_{Y}$ and $g':=h|_{Z}$ have the desired properties.
Now it follows from \cite[Theorem~1.4]{DraWes92IntersectUnstably} that $\dim(Y\times Z)\leq n-1$.

Let $D:=\{(x,x): X\in X\}$ denote the diagonal in $X\times X$.
The complement $(X\times X)\setminus D$ can be covered by countably many compact rectangular sets, that is, there exist closed subsets $Y_k,Z_k\subseteq X$ for $k\in\mathbb{N}$ such that
\[
(X\times X)\setminus D = \bigcup_{k\in\mathbb{N}} (Y_k\times Z_k).
\]
For each $k$, note that $Y_k$ and $Z_k$ are disjoint and therefore $\dim(Y_k\times Z_k)\leq n-1$ by the Claim.
Note also that $X$ is covered by the sets $(Y_k)_{k\in\mathbb{N}}$.
Using the countable sum theorem (\cite[Theorem~3.2.5, p.125]{Pea75DimThy}), we get 
\[
\dim(D)
=\dim(X)
=\sup_{k\in\mathbb{N}}\dim(Y_k)
\leq \sup_{k\in\mathbb{N}}\dim(Y_k\times Z_k)
\leq n-1.
\]
Applying the countable sum theorem again, we obtain
\[
\dim(X\times X) 
= \max\{\dim(D),\sup_{k\in\mathbb{N}}\dim(Y_k\times Z_k)\}
\leq n-1. \qedhere
\]
\end{proof}

\begin{lma}
\label{prp:gr_commutative_sep}
Let $X$ be a compact, metric space.
Then $\gr(C(X))=\dim(X\times X)$.
\end{lma}
\begin{proof}
Let $n\in\mathbb{N}$.
As explained at the beginning of the section, we have $\gr(C(X))\leq n$ if and only if the following two conditions hold:
\begin{enumerate}[(1)  ]
\item
$C(X,\mathbb{R}^{n+1}\setminus\{0\})\subseteq C(X,\mathbb{R}^{n+1})$ is dense;
\item
$E(X,\mathbb{R}^{n+1}\setminus\{0\})\subseteq C(X,\mathbb{R}^{n+1}\setminus\{0\})$ is dense.
\end{enumerate}

By \cite[Proposition~10.3.2, p.378]{Pea75DimThy}, we have $\dim(X)\leq n$ if and only if $C(X,\mathbb{R}^{n+1}\setminus\{0\})\subseteq C(X,\mathbb{R}^{n+1})$ is dense.
Further, by \autoref{prp:approxByEmb}, we have $\dim(X\times X)\leq n$ if and only if $E(X,\mathbb{R}^{n+1})\subseteq C(X,\mathbb{R}^{n+1})$ is dense, which in turn is equivalent to $E(X,\mathbb{R}^{n+1}\setminus\{0\})\subseteq C(X,\mathbb{R}^{n+1}\setminus\{0\})$ being dense.

We deduce that $\gr(C(X))\leq n$ if and only if $\dim(X)\leq n$ and $\dim(X\times X)\leq n$.
Since $\dim(X)\leq\dim(X\times X)$, the result follows.
\end{proof}

Let $X$ be a finite-dimensional, compact, metric space.
In \cite[Theorem~3.16]{Dra01CohomDimThy}, Dranishnikov showed that the dimension of $X\times X$ can only take the values $2\dim(X)$ or $2\dim(X)-1$.
If $\dim(X\times X)=2\dim(X)$, then $X$ is said to be of \emph{basic type}, and then $\dim(X^k)=k\dim(X)$ for every $k\geq 1$, where $X^k=X\times\ldots_k\times X$ is the $k$-fold Cartesian power.
If $\dim(X\times X)=2\dim(X)-1$, then $X$ is said to be of \emph{exceptional type}, and then $\dim(X^k)=k\dim(X)-k+1$ for every $k\geq 1$.

If $\dim(X)\leq 1$, then $X$ is of basic type.
The prime example of a space of exceptional type was constructed by Boltyanski\u{\i} in \cite{Bol49TwoDimQuareThreeDim}:
It is a compact, metric space $P$ with $\dim(P)=2$ and $\dim(P^2)=3$.
For each $k\geq 1$, it follows that $\dim(P^k)=k+1$ and $\dim((P^k)^2)=2(k+1)-1$, and thus $P^k$ is a $(k+1)$-dimensional compact, metric space of exceptional type.

Next, we want to generalize \autoref{prp:gr_commutative_sep} to compact, Hausdorff spaces.
The basic idea is to approximate a given compact, Hausdorff space $X$ by suitable compact, metric spaces $X_\lambda$.
It is a standard technique to approximate $X$ by polyhedra $X_\lambda$ with $\dim(X_\lambda)=\dim(X)$.
However, this does not ensure that $\dim(X_\lambda\times X_\lambda)=\dim(X\times X)$, and indeed, polyhedra are always of basic type, so if $X$ is of exceptional type than an approximation of $X$ by polyhedra is not suitable.
To address this, we develop a different way of approximation.
The next result also shows that the dichotomy of basic/exceptional type for compact, metric spaces generalizes to compact, Hausdorff spaces, which will be crucial in the proof of \autoref{prp:dim_space_squared}.

\begin{prp}
\label{prp:typeOfSpace}
Let $X$ be a finite-dimensional, compact, Hausdorff space.
Set $n:=\dim(X)$.
Then either $\dim(X\times X)=2n$ (we say that $X$ is of \emph{basic type}) and then $\dim(X^k)=kn$ for every $k\geq 1$;
or $\dim(X\times X)=2n-1$ (we say that $X$ is of \emph{exceptional type}) and then $\dim(X^k)=kn-k+1$ for every $k\geq 1$.
\end{prp}
\begin{proof}
By the product theorem for covering dimension, we have $\dim(Y\times Z)\leq\dim(Y)+\dim(Z)$ for any compact, Hausdorff spaces $Y$ and $Z$; 
see \cite[Proposition~9.3.2, p.352]{Pea75DimThy}.
Thus, $\dim(X^k)\leq kn$ for every $k\geq 1$.
The statement follows from the following claims.

\emph{Claim~1: Let $k\geq 2$ such that $\dim(X^k)<kn$.
Then $\dim(X^l)\leq ln-l+1$ for all $l\geq 1$.}
To prove the claim, set $A:=C(X)$. 
Then $A^{\otimes k}\cong C(X^k)$, and so $\rr(A)=\dim(X)=n$ and $\rr(A^{\otimes k})<kn$.
Consider
\[
\mathcal{S} := \big\{ B\subseteq A \text{ unital, separable sub-\ca} : \rr(B)\leq n, \rr(B^{\otimes k})<kn \big\}.
\]
Given $B\in\mathcal{S}$, let $Y$ be a compact, metric space such that $B\cong C(Y)$.
Then $\dim(Y)=\rr(B)\leq n$ and $\dim(Y^k)=\rr(B^{\otimes k})<kn$.
If $\dim(Y)=n$, then $Y$ is of exceptional type and we get $\dim(Y^l)=ln-l+1$ for all $l\geq 1$.
If $\dim(Y)\leq n-1$, then $\dim(Y^l)\leq l(n-1)$.
In either case $\rr(B^{\otimes l})=\dim(Y^l)\leq ln-l+1$ for all $l\geq 1$.

Since unitality and the properties `$\rr(\freeVar)\leq n$' and `$\rr(\freeVar^{\otimes k})<kn$' satisfy the L\"{o}wenheim-Skolem condition (see \autoref{pgr:LS} and \autoref{prp:rrTensLS}), it follows that the family $\mathcal{S}$ approximates $A$.
Hence, for each $l\geq 1$, $A^{\otimes l}$ is approximted by sub-\ca{s} of the form $B^{\otimes l}$ with $\rr(B^{\otimes l})\leq ln-l+1$.
Since the real rank satisfies \axiomD{5}, we get $\dim(X^l)=\rr(A^{\otimes l})\leq ln-l+1$.

\emph{Claim~2: Let $k\geq 2$. Then $\dim(X^k)\geq kn-k+1$.}
To prove the claim via contradiction, asssume that $k\geq 2$ is such that $\dim(X^k)< kn-k+1$.
Consider
\[
\mathcal{S} := \big\{ B\subseteq A \text{ unital, separable sub-\ca} : \rr(B^{\otimes k})<kn-k+1 \big\}.
\]
Given $B\in\mathcal{S}$, let $Y$ be a compact, metric space such that $B\cong C(Y)$.
Then $\dim(Y^k)=\rr(B^{\otimes k})<kn-k+1$.
It follows that $\dim(Y)\leq n-1$.

As in Case~1, using that unitality and `$\rr(\freeVar^{\otimes k})<kn-k+1$' satisfy the L\"{o}wenheim-Skolem condition, we obtain that $\mathcal{S}$ approximates $A$. 
Since the real rank satisfies \axiomD{5}, we get $\dim(X)=\rr(A)\leq n-1$, a contradiction.
\end{proof}

\begin{prp}
\label{prp:gr_commutative_unit}
Let $X$ be a compact, Hausdorff space.
Then
\[
\gr(C(X))=\dim(X\times X).
\]
\end{prp}
\begin{proof}
Since $\dim(X\times X)=\rr(C(X)\otimes C(X))$, and since $\gr(\freeVar)=\grPre(\freeVar)$ for unital \ca{s}, it suffices to show that $\grPre(A)=\rr(A^{\otimes 2})$ for every unital, commutative \ca{} $A$.
By \autoref{prp:gr_commutative_sep}, we have $\grPre(B)=\rr(B^{\otimes 2})$ for every \emph{separable}, unital, commutative \ca{} $B$.

Let $A$ be a unital, commutative \ca{}, and set $n:=\grPre(A)$.
Since unitality and `$\rr(\freeVar^{\otimes 2})\leq n$' satisfy the L\"{o}wenheim-Skolem condition (see \autoref{pgr:LS} and \autoref{prp:rrTensLS}), there exists a $\sigma$-complete, cofinal family $\mathcal{S}\subseteq\SubSep(A)$ such that each $B\in\mathcal{S}$ is a separable, unital, commutative \ca{s} satisfying $\rr(B^{\otimes 2})\leq n$.
Then $\grPre(B)\leq n$.
Since $\grPre$ satisfies \axiomD{6} by \autoref{prp:gr0_approx}, we obtain that $\grPre(A)\leq n=\rr(A\otimes A)$.

The converse inequality follows analogously, using that unitality and `$\grPre(\freeVar)\leq n$' satisfy the L\"{o}wenheim-Skolem condition, and using that $\rr(\freeVar\otimes\freeVar)$ satisfies \axiomD{6}.
\end{proof}

\begin{pgr}
\label{pgr:locdim}
Let $X$ be a locally compact, Hausdorff space.
If $X$ is not compact, then $C_0(X)$ is not unital, and the minimal unitization of $C_0(X)$ is naturally isomorphic to $C(\alpha X)$, where $\alpha X$ denotes the one-point compactification of $X$.
As noted in \cite[2.2(ii)]{BroPed09Limits}, it follows from \cite[Proposition~3.5.6]{Pea75DimThy} that the dimension of $\alpha X$ agrees with the \emph{local dimension} of $X$, which is defined as
\[
\locdim(X) := \sup \big\{ \dim(K) : K\subseteq X \text{ compact } \big\}.
\]
A similar argument shows that
\[
\dim(\alpha X \times \alpha X) = \locdim(X\times X).
\]
\end{pgr}

\begin{thm}
\label{prp:gr_commutative}
Let $X$ be a locally compact, Hausdorff space.
Then
\[
\grPre(C_0(X))
= \gr(C_0(X))
= \locdim(X\times X).
\]
\end{thm}
\begin{proof}
By \autoref{prp:grPre_less_gr}, we have $\grPre(C_0(X)) \leq \gr(C_0(X))$.
Let $K\subseteq X$ be a compact subset.
Then $C(K)$ is a quotient of $C_0(X)$.
Hence, using \autoref{prp:gr_commutative_unit} at the first step, and using \autoref{prp:gr0_quotients} at the last step, we get
\[
\dim(K\times K)
= \gr(C(K))
= \grPre(C(K))
\leq \grPre(C_0(X)).
\]
Since every compact subset of $X\times X$ is contained in $K\times K$ for some compact subset $K\subseteq X$, we deduce that
\[
\locdim(X\times X) \leq \grPre(C_0(X)) \leq \gr(C_0(X)).
\]
Finally, using \autoref{prp:gr_commutative_unit}, we have
\[
\gr(C_0(X))
= \gr(C(\alpha X))
= \dim(\alpha X\times \alpha X)
= \locdim(X\times X).\qedhere
\] 
\end{proof}

\begin{prp}
\label{prp:charGr0}
Let $A$ be a \ca{}.
Then the following are equivalent:
\begin{enumerate}
\item
We have $\gr(A)=0$;
\item
We have $\grPre(A)=0$;
\item
$A$ is a commutative \ca{} with totally disconnected spectrum.
\end{enumerate}
\end{prp}
\begin{proof}
\emph{Claim: Let $X$ be a locally compact, Hausdorff space.
Then $X$ is totally disconnected if and only if $\locdim(X\times X)=0$.}
It is well-known that $X$ is totally disconnected if and only if $X$ is zero-dimensional (that is, the topology of $X$ has a basis of clopen sets), which in turn is equivalent to $\locdim(X)=0$.
If a compact, Hausdorff space $Y$ satisfies $\dim(Y)=0$, then $\dim(Y\times Y)\leq 2\dim(Y)=0$ by the product theorem for covering dimension (\cite[Theorem~9.3.2, p.351]{Pea75DimThy}).
The converse also holds, and thus $Y$ satisfies $\dim(Y)=0$ if and only if $\dim(Y\times Y)=0$.
We deduce that $\locdim(X)=0$ if and only if $\locdim(X\times X)=0$, which proves the claim.

By \autoref{prp:grPre_less_gr}, we have $\grPre(A)\leq\gr(A)$, which shows that~(1) implies~(2).
Assuming~(2), let us verify~(3). 
Let $a,b\in A$.
Apply \autoref{prp:building_generators} (for any $x$ and $\varepsilon$) to obtain $y\in A_\sa$ such that $a,b\in C^*(y)$.
Since $y$ is self-adjoint, $C^*(y)$ is commutative, and thus $a$ and $b$ commute.
Since this holds for all $a,b\in A$, we deduce that $A$ is commutative.
Let $X$ be a locally compact, Hausdorff space such that $A\cong C_0(X)$.
By \autoref{prp:gr_commutative}, we have $\locdim(X\times X)=\gr(A)=0$, and thus $X$ is totally disconnected by the Claim.

Finally, assuming~(3), let us prove~(1).
Let $A=C_0(X)$ for a totally disconnected, locally compact Hausdorff space $X$.
Using the claim, we get $\locdim(X\times X)=0$, and thus $\gr(A)=0$ by \autoref{prp:gr_commutative}.
\end{proof}

\begin{lma}
\label{prp:dim_space_squared}
Let $X, Y$ be locally compact, Hausdorff spaces, and let $Z:=X\sqcup Y$ denote the disjoint union.
Then
\[
\locdim(Z\times Z) = \max\big\{ \locdim(X\times X), \locdim(Y\times Y) \big\}.
\]
\end{lma}
\begin{proof}
The inequality `$\geq$' follows since $X\times X$ and $Y\times Y$ are homeomorphic to closed subsets of $Z\times Z$.
To show the converse inequality, let $K\subseteq Z\times Z$ be a compact subset.
Then there exist compact subsets $X'\subseteq X$ and $Y'\subseteq Y$ such that $K$ is contained in $(X'\cup Y')\times (X'\cup Y')$.
If we can show the inequality for $X'$ and $Y'$, then
\begin{align*}
\dim(K)
&\leq\dim((X'\cup Y')\times (X'\cup Y')) 
\leq \max\big\{ \dim(X'\times X'), \dim(Y'\times Y') \big\} \\
&\leq \max\big\{ \locdim(X\times X), \locdim(Y\times Y) \big\}.
\end{align*}

Thus, without loss of generality, we may assume that $X$ and $Y$ are compact.
We need to show that $\dim(Z^2)\leq\max\{\dim(X^2), \dim(Y^2)\}$.
Note that $\dim(Z)=\max\{\dim(X),\dim(Y)\}$.
We distinguish two cases:

\emph{Case~1: Assume that $\dim(X)\neq\dim(Y)$.}
Without loss of generality, we have $\dim(X)<\dim(Y)$.
Then
\[
\dim(X\times Y)
\leq \dim(X)+\dim(Y)
\leq 2\dim(Y)-1
\leq \dim(Y^2).
\]
Since $Z^2$ is the union of sets homeomorphic to $X^2$, $X\times Y$ and $Y^2$, we get
\[
\dim(Z^2)
= \max \big\{ \dim(X^2),\dim(X\times Y),\dim(Y^2) \big\}
\leq \max \big\{ \dim(X^2),\dim(Y^2) \big\}.
\]

\emph{Case~2: Assume that $\dim(X)=\dim(Y)$.}
Set $d:=\dim(X)$.
Then $\dim(Z)=d$ and so $\dim(Z^2)\leq 2d$.
If $X$ or $Y$ is of basic type, then 
\begin{align*}
\dim(Z^2)
\leq 2d
= \max \big\{ \dim(X^2), \dim(Y^2) \big\}.
\end{align*}

If $X$ and $Y$ are of exceptional type, then $\dim(X^2)=\dim(Y^2)=2d-1$.
Hence
\[
\dim(X^2\times Y) \leq \dim(X^2)+\dim(Y) = 3d-1,
\]
and similarly $\dim(X\times Y^2)\leq 3d-1$.
Using that $Z^3$ is the union of space homeomorphic to $X^3$, $X^2\times Y$, $X\times Y^2$ and $Y^3$, we deduce that
\begin{align*}
\dim(Z^3)
&= \max \big\{ \dim(X^3),\dim(X^2\times Y),\dim(X\times Y^2),\dim(Y^3) \big\} \\
&\leq\max\{3d-2,3d-1,3d-1,3d-2\}
\leq 3d-1.
\end{align*}
If $Z$ were of basic type, then $\dim(Z^3)=3\dim(Z)$.
Thus, $Z$ is of exceptional type, and so $\dim(Z^2)=2\dim(Z)-1=2d-1 = \max\{\dim(X^2),\dim(Y^2)\}$.
\end{proof}

\begin{prp}
\label{prp:gr_sum_commutative}
Let $A$ and $B$ be commutative \ca{s}.
Then
\[
\gr(A\oplus B)=\max \big\{ \gr(A),\gr(B) \big\}.
\]
\end{prp}
\begin{proof}
Let $X$ and $Y$ be locally compact, Hausdorff spaces such that $A\cong C_0(X)$ and $B\cong C_0(Y)$.
Set $Z:=X\sqcup Y$.
Then $A\oplus B\cong C_0(Z)$.
Using \autoref{prp:gr_commutative} at the first and last step, and using \autoref{prp:dim_space_squared} at the second step, we get
\begin{align*}
\gr(A\oplus B)
&= \locdim(Z\times Z)
= \max\big\{ \locdim(X\times X), \locdim(Y\times Y) \big\} \\
&= \max\big\{ \gr(A), \gr(B) \big\}. \qedhere
\end{align*}
\end{proof}

The previous result shows that the generator rank behaves well with respect to sums when both summands are commutative.
Next, we generalize this to the case that only one summand is commutative.
We first recall a standard result about the structure of subalgebras of direct summands.

\begin{lma}
\label{prp:subSum}
Let $A_1,A_2$ be \ca{s}, let $\pi_j\colon A_1\oplus A_2\to A_j$ be the quotient maps onto the summands, for $j=1,2$, and let $B\subseteq A_1\oplus A_2$ be a sub-\ca{} such that $\pi_j(B)=A_j$ for $j=1,2$.
Then $I_1:=\{a_1\in A_1 : (a_1,0)\in B\}$ is an ideal in $A_1$, and $I_2:=\{a_2\in A_2 : (0,a_2)\in B\}$ is an ideal in $A_2$.
Further, there is an isomorphism $\alpha\colon A_1/I_1\to A_2/I_2$ such that
\[
B = \big\{ (a_1,a_2) \in A_1\oplus A_2 : \alpha(a_1+I_1)=a_2+I_2 \big\}.
\]
\end{lma}

\begin{prp}
\label{prp:specialSums}
Let $A$ and $B$ be separable \ca{s}.
Let $I\subseteq A$ be an ideal such that whenever $J\subseteq A$ is an ideal such that $A/J$ is isomorphic to a quotient of $B$, then $I\subseteq J$.
Then $\grPre(A\oplus B)=\max\{\grPre(A),\grPre((A/I) \oplus B)\}$.
\end{prp}
\begin{proof}
The inequality `$\geq$' follows from \autoref{prp:gr0_quotients} using that $A$ and $(A/I)\oplus B$ are quotients of the direct sum $A\oplus B$.
To show the converse inequality, set $n:=\max\{\grPre(A),\grPre((A/I) \oplus B)\}$, which we may assume to be finite.
We will show that $\Gen_{n+1}(A\oplus B)_\sa$ is dense in $(A\oplus B)^{n+1}_\sa$.

Since $\grPre(A)\leq n$, by \autoref{prp:grSep} $\Gen_{n+1}(A)_\sa$ is a dense $G_\delta$-subset of $A^{n+1}_\sa$.
It follows that
\[
D_1 := \big\{ (\vect{a},\vect{b})\in (A\oplus B)^{n+1}_\sa : \vect{a}\in\Gen_{n+1}(A)_\sa \big\}
\]
is a dense $G_\delta$-subset of $(A\oplus B)^{n+1}_\sa$.

Let $\pi\colon A\to A/I$ denote the quotient map, which naturally induces a map $A^{n+1}\to(A/I)^{n+1}$, which we also denote by $\pi$.
Since $\grPre((A/I)\oplus B)\leq n$, $\Gen_{n+1}((A/I)\oplus B)_\sa$ is a dense $G_\delta$-subset of $((A/I)\oplus B)^{n+1}_\sa$.
It follows that
\[
D_2 := \big\{ (\vect{a},\vect{b})\in (A\oplus B)^{n+1}_\sa : (\pi(\vect{a}),\vect{b})\in\Gen_{n+1}((A/I)\oplus B)_\sa \big\}
\]
is a dense $G_\delta$-subset of $(A\oplus B)^{n+1}_\sa$.
It follows from the Baire category theorem that $D:=D_1\cap D_2$ is a dense subset of $(A\oplus B)^{n+1}_\sa$.

We show that $D\subseteq\Gen_{n+1}(A\oplus B)_\sa$, which will finish the proof.
Let $(\vect{a},\vect{b})\in D$, and set $E:=C^*((\vect{a},\vect{b}))\subseteq A\oplus B$.
Let $\pi_A$ and $\pi_B$ denote the quotient maps from $A\oplus B$ onto $A$ and $B$, respectively.
Further, let $\pi'\colon A\oplus B\to (A/I)\oplus B$ denote the quotient map.
Using that $(\vect{a},\vect{b})\in D_1$, it follows that $\pi_A(E)=A$.
Similarly, we obtain that $\pi'(E)=(A/I)\oplus B$, and so $\pi_B(E)=B$.

By \autoref{prp:subSum}, there are ideals $J_1\subseteq A$ and $J_2\subseteq B$ and an isomorphism $\alpha\colon A/J_1\to B/J_2$ such that
\[
E = \big\{ (a,b)\in A\oplus B : \alpha(a+J_1)=b+J_2 \big\}.
\]
Assuming that $A/J_1\neq\{0\}$, it follows that $E/(J_1\oplus J_2)$ is a proper subset of $(A\oplus B)/(J_1\oplus J_2)$.
However, by assumption we have $I\subseteq J_1$.
Hence, $I\oplus 0$ is a smaller ideal than $J_1\oplus J_2$.
Since $\pi'(E)=(A/I)\oplus B$, we have that $E/(I\oplus 0)=(A\oplus B)/(I\oplus 0)$, which leads to a contradiction.
Thus, $A/J_1=\{0\}$, which implies that $E=A\oplus B$.
\end{proof}

\begin{prp}
\label{prp:gr_sumWithCommutative}
Let $A$ and $B$ be \ca{s} and assume that $B$ is commutative.
Then:
\[
\grPre(A\oplus B)=\max\{\grPre(A),\grPre(B)\}, \quad\text{ and }\quad
\gr(A\oplus B)=\max\{\gr(A),\gr(B)\}.
\]
\end{prp}
\begin{proof}
\emph{Step~1:
We prove the statement for $\grPre$ under the additional assumption that $A$ and $B$ are separable.
}
In this case, let $I\subseteq A$ be the smallest ideal such that $A/I$ is commutative.
(Such an ideal exists, and $A/I$ is the abelianization of $A$.)
Then the assumptions of \autoref{prp:specialSums} are satisfied, and we obtain
\[
\grPre(A\oplus B)
= \max \big\{ \grPre(A),\grPre((A/I) \oplus B) \big\}.
\]
Since $A/I$ and $B$ are commutative, applying Propositions~\ref{prp:gr0_quotients}  and~\ref{prp:gr_sum_commutative} we get
\[
\grPre((A/I) \oplus B)
= \max\big\{ \grPre(A/I),\grPre(B) \big\}
\leq \max\big\{ \grPre(A),\grPre(B) \big\},
\]
and thus $\grPre(A\oplus B)\leq\max\{\grPre(A),\grPre(B)\}$.
The converse inequality follows from \autoref{prp:gr0_quotients} using that $A$ and $B$ are quotients of $A\oplus B$.

\emph{Step~2:
We prove statement for $\grPre$ for general \ca{s}.
}
Set
\[
\mathcal{S}_1 := \{A'\subseteq A \text{ separable} : \grPre(A')\leq\grPre(A) \big\}.
\]
Using that $\grPre$ satisfies \axiomD{6}, the family $\mathcal{S}_1$ approximates $A$.
Similarly, $B$ is approximated by the family $\mathcal{S}_2$ of separable sub-\ca{s} $B'\subseteq B$ that satisfy $\grPre(B')\leq\grPre(B)$.
Given $A'\in\mathcal{S}_1$ and $B'\in\mathcal{S}_2$, note that $B'$ is automatically commutative, whence by Step~1 we obtain
\[
\grPre(A'\oplus B')
= \max \big\{ \grPre(A'),\grPre(B') \big\}
\leq \max \big\{ \grPre(A),\grPre(B) \big\}.
\]
Since $A\oplus B$ is approximated by such $A'\oplus B'$, and since $\grPre$ satisfies \axiomD{5}, we deduce that $\gr(A\oplus B)\leq\max\{ \grPre(A),\grPre(B) \}$.
The converse inequality is clear.

\emph{Step~3:
We prove statement for $\gr$.
}
Applying \autoref{prp:gr_grPreMaxRR} at the first and last step, and using Step~2 and that the real rank behaves well with respect to direct sums (since the real rank is a noncommutative dimension theory as noted in \cite[Remark~2]{Thi13TopDimTypeI}), we have
\begin{align*}
\gr(A\oplus B)
&= \max \big\{ \grPre(A\oplus B), \rr(A\oplus B) \big\} \\
&= \max \big\{ \grPre(A), \grPre(B), \rr(A), \rr(B) \big\} 
= \max \big\{ \gr(A), \gr(B) \big\}. \qedhere
\end{align*}
\end{proof}

\section{Permanence properties}
\label{sec:permanence}

In this section, we show that the generator rank enjoys the same permanence properties as its precursor $\grPre$.

Recall that $A^+$ denotes the forced unitization of a \ca{} $A$.

\begin{lma}
\label{prp:gr_forcedUnit}
Let $A$ be a \ca.
Then $\gr(A^+)=\gr(A)$.
\end{lma}
\begin{proof}
If $A$ is nonunital, then $A^+=\widetilde{A}$ and the result follows from the definition.
If $A$ is unital, then $A^+=A\oplus\mathbb{C}$, and by \autoref{prp:gr_sumWithCommutative} we have
\[
\gr(A^+)
= \gr(A\oplus\mathbb{C}) 
= \max\big\{ \gr(A), \gr(\mathbb{C}) \big\}
= \gr(A). \qedhere
\]
\end{proof}

\begin{thm}
\label{prp:gr_idealQuotExt}
Let $A$ be a \ca{}, and let $I\subseteq A$ be an ideal.
Then:
\[
\max\big\{ \gr(I),\gr(A/I) \big\} \leq \gr(A) \leq \gr(I)+\gr(A/I)+1.
\]
\end{thm}
\begin{proof}
It is known that the real rank does not increase when passing to ideals or quotients;
see \cite[Th\'{e}or\`{a}me~1.4]{Elh95RRExt}.
By Propositions~\ref{prp:gr0_ideals} and~\ref{prp:gr0_quotients}, $\grPre$ does not increase when passing to ideals or quotients.
Applying \autoref{prp:gr_grPreMaxRR}, we get
\[
\gr(I)
= \max \big\{ \rr(I), \grPre(I) \big\}
\leq \max \big\{ \rr(A), \grPre(A) \big\}
= \gr(A),
\]
and analogously $\gr(A/I)\leq\gr(A)$.

To verify $\gr(A) \leq \gr(I)+\gr(A/I)+1$, we consider $I$ as an ideal in $\widetilde{A}$.
Note that $\widetilde{A}/I$ is naturally isomorphic to $(A/I)^+$.
Using \autoref{prp:gr_forcedUnit}, we obtain
\[
\grPre(\widetilde{A}/I)
= \grPre((A/I)^+)
= \gr((A/I)^+)
= \gr(A/I).
\]

Applying \autoref{prp:gr0_extensions} and that $\grPre(I)\leq\gr(I)$ by \autoref{prp:grPre_less_gr}, we get
\[
\gr(A)
= \grPre(\widetilde{A})
\leq \grPre(I)+\grPre(\widetilde{A}/I)+1
\leq \gr(I)+\gr(A/I)+1. \qedhere
\]
\end{proof}

\begin{thm}
\label{prp:gr_approx_limits}
Let $A$ be a \ca{} and $n\in\mathbb{N}$.
Assume that $A$ is approximated by sub-\ca{s} $A_\lambda\subseteq A$ with $\gr(A_\lambda)\leq n$ for each $\lambda$.
Then $\gr(A)\leq n$.

Moreover, if $A=\varinjlim_{j} A_j$ is an inductive limit, then $\gr(A)\leq\liminf_{j} \gr(A_j)$.
\end{thm}
\begin{proof}
\emph{Case~1: We assume that $A$ is unital.}
By \autoref{prp:grPre_less_gr}, we have $\grPre(A_\lambda)\leq\gr(A_\lambda)\leq n$ for each $\lambda$.
Using \autoref{prp:gr0_approx}, we obtain $\gr(A)=\grPre(A)\leq n$.

\emph{Case~2: We assume that $A$ is nonunital.}
In this case, for each $\lambda$, the force unitization $A_\lambda^+$ is a unital sub-\ca{} of $\widetilde{A}$, and the family $(A_\lambda^+)_\lambda$ approximates~$\widetilde{A}$.
By \autoref{prp:gr_forcedUnit}, we have $\gr(A_\lambda^+)=\gr(A_\lambda)\leq n$ for each $\lambda$.
Now the result follows from Case~1.

The statement for inductive limits is proved analogously to \autoref{prp:gr0_limits}, using that the generator rank does not increase when passing to quotients, and using the result for approximation by subalgebras.
\end{proof}

Since we do not know how to compute $\grPre$ of direct sums (\autoref{qst:gr0_sums}), the analog for the generator rank is also unclear:

\begin{qst}
\label{qst:gr_sums}
Do we have $\gr(A\oplus B)=\max\{\gr(A),\gr(B)\}$ for all $A$ and $B$?
\end{qst}

We provide a positive answer to \autoref{qst:gr_sums} for the case that both algebras have real rank zero (see \autoref{prp:gr_sum_rr0}), and for the case that one of the algebras is commutative (see \autoref{prp:gr_sumWithCommutative}).
We note that already the case that both algebras are commutative is surprisingly nontrivial;
see \autoref{prp:gr_sum_commutative}.

\section{Generator rank of AF-algebras}
\label{sec:AF}

In this section, we show that every finite-dimensional \ca{} has generator rank at most one;
see \autoref{prp:gr_fd}.
Using that the generator rank behaves well with respect to inductive limits, we deduce the main result:
AF-algebras have generator rank at most one;
see \autoref{prp:gr_AF-alg}.

\begin{lma}
\label{prp:gr_sum_rr0}
Let $A, B$ be \ca{s} of real rank zero.
Then $\gr(A\oplus B)=\max\{\gr(A),\gr(B)\}$.
\end{lma}
\begin{proof}
Since $A$ and $B$ are quotients of $A\oplus B$, it follows from \autoref{prp:gr_idealQuotExt} that $\gr(A\oplus B)\geq\max\{\gr(A),\gr(B)\}$.
It remains to verify the converse inequality.
Since $A$ has real rank zero, we have $\gr(A)=\grPre(A)$ by \autoref{prp:gr_grPreMaxRR}.
Similarly, $\gr(B)=\grPre(B)$ and $\gr(A\oplus B)=\grPre(A\oplus B)$.
Thus, it suffices to prove $\grPre(A\oplus B)\leq\max\{\grPre(A),\grPre(B)\}$.

Set $n:=\max\{\grPre(A),\grPre(B)\}$, which we may assume to be finite.
To verify $\grPre(A\oplus B)\leq n$, let $\vect{a}\in A^{n+1}_\sa$, $\vect{b}\in B^{n+1}_\sa$, $\varepsilon>0$, $x\in A$, and $y\in B$.
We need to find $\vect{c}\in A^{n+1}_\sa$ and $\vect{d}\in B^{n+1}_\sa$ such that
\[
\vect{c}=_\varepsilon\vect{a}, \quad
\vect{d}=_\varepsilon\vect{b}, \quad\text{ and }\quad
(x,y)\in_\varepsilon C^*((\vect{c},\vect{d})).
\]
Since $\rr(A)=\rr(B)=0$, we obtain $a_0'\in A_\sa$ and $b_0'\in B_\sa$ such that $a_0'=_{\varepsilon/2} a_0$, $b_0'=_{\varepsilon/2} b_0$ and such that the spectra $\sigma(a_0')$ and $\sigma(b_0')$ are finite, disjoint and do not contain $0$.
Let $\delta_0>0$ be smaller than the distance between any two points in $\sigma(a_0')\cup\sigma(b_0')\cup\{0\}$.
Let $f,g\colon\mathbb{R}\to[0,1]$ be continuous functions such that:
\begin{enumerate}
\item
$f$ takes value $1$ on the $\delta_0/4$-neighborhood of $\sigma(a_1')$, and takes value $0$ on the $\delta_0/4$-neighborhood of $\sigma(b_1')\cup\{0\}$.
\item
$g$ takes value $1$ on the $\delta_0/4$-neighborhood of $\sigma(b_1')$, and takes value $0$ on the $\delta_0/4$-neighborhood of $\sigma(a_1')\cup\{0\}$.
\end{enumerate}

Choose $\delta>0$ with $\delta<\varepsilon/2$ and such that:
\begin{enumerate}
\item
Whenever $c_0\in A_\sa$ satisfies $c_0=_\delta a_0'$,
then the spectrum $\sigma(c_0)$ is contained in the $\delta_0/4$-neighborhood of $\sigma(a_0')$.
\item
Whenever $d_0\in B_\sa$ satisfies $d_0=_\delta b_0'$,
then the spectrum $\sigma(d_0)$ is contained in the $\delta_0/4$-neighborhood of $\sigma(b_0')$.
\end{enumerate}

Using that $\grPre(A), \grPre(B)\leq n$, we obtain $\vect{c}\in A^{n+1}_\sa$, $\vect{d}\in B^{n+1}_\sa$ and nc-poly\-no\-mi\-als $p$ and $q$ such that
\[
\vect{c}=_\delta(a_0',a_1,\ldots,a_n), \quad
x=_\varepsilon p(\vect{c}), \quad
\vect{d}=_\delta(b_0',b_1,\ldots,b_n), \quad\text{ and }\quad
y=_\varepsilon q(\vect{d}).
\]
Then $(\vect{c},\vect{d})=_\varepsilon(\vect{a},\vect{b})$, and we will show that $(x,y)\in_\varepsilon C^*((\vect{c},\vect{d}))$.

By choice of $\delta$, we have $f(c_0,d_0)=(1,0)$ and $g(c_0,d_0)=(0,1)$.
Then:
\[
(x,y)
=_\varepsilon (p(\vect{c}),q(\vect{d}))
= f(c_0, d_0) p(\vect{c},\vect{d}) + g(c_0, d_0) q(\vect{c},\vect{d}) 
\in C^*((\vect{c},\vect{d})). \qedhere
\]
\end{proof}

\begin{lma}
\label{prp:gr_fd}
Let $A$ be a finite-dimensional \ca.
Then $\gr(A)\leq 1$.
\end{lma}
\begin{proof}
The statement follows from \autoref{prp:gr_sum_rr0} once we show that every matrix algebra has generator rank at most one.

Let $d\geq 1$.
To show $\gr(M_d(\mathbb{C}))\leq 1$, let $a,b\in M_d(\mathbb{C})_\sa$ and $\varepsilon>0$.
We will approximate $a$ and $b$ by self-adjoint elements that generate $M_d(\mathbb{C})$.
Choose a unitary $u$ such that $uau^*$ is diagonal, and then choose a diagonal, self-adjoint element $a'\in M_d(\mathbb{C})$ such that $a'=_\varepsilon uau^*$ and such that the spectrum of $a'$ contains $d$ different nonzero entries.
Next, choose $b'\in M_d(\mathbb{C})_\sa$ such that $b'=_\varepsilon ubu^*$ and such that all entries in the matrix $b'$ are nonzero.

Let $e_{jk}\in M_d(\mathbb{C})$ denote the matrix units for $j,k\in\{1,\ldots,d\}$.
Let $t_1,\ldots,t_d$ be the pairwise distinct, nonzero diagonal entries of~$a'$.
Given $k\in\{1,\ldots,d\}$, let $f_k\colon\mathbb{R}\to[0,1]$ be a continuous function that vanishes on $0$ and on $t_j$ for $j\neq k$, and that satisfies $f_k(t_k)=1$.
Applying functional calculus, we obtain $e_{kk}=f_k(a')\in C^*(a',b')$.
Given $j,k\in\{1,\ldots,d\}$, we have $b_{jk}'e_{jk}=e_{jj}b'e_{kk}\in C^*(a',b')$, where $b_{jk}'$ is the (nonzero) $jk$-entry of $b'$.
It follows that $C^*(a',b')$ contains all matrix units, and thus $C^*(a',b')=M_d(\mathbb{C})$.

We have $u^*a'u=_\varepsilon a$ and $u^*b'u=_\varepsilon b$.
Moreover, $u^*a'u$ and $u^*b'u$ generate $M_d(\mathbb{C})$ since $C^*(u^*a'u,u^*b'u)=u^*C^*(a',b')u=M_d(\mathbb{C})$.
\end{proof}

Recall that a \ca{} $A$ is said to be \emph{locally finite-dimensional} if for every finite subset $F\subseteq A$ and $\varepsilon>0$ there exists a finite-dimensional sub-\ca{} $B\subseteq A$ such that $x\in_\varepsilon B$ for every $x\in F$.
Further, $A$ is \emph{approximately finite-dimensional}, or an \emph{AF-algebra}, if $A$ is isomorphic to an inductive limit of finite-dimensional \ca{s}.
Equivalently, $A$ contains a directed family of finite-dimensional sub-\ca{s} $A_\lambda\subseteq A$ such that $\bigcup_\lambda A_\lambda$ is dense in $A$.
Every AF-algebra is locally finite-dimensional, and the converse holds for separable \ca{s}, but not in general;
see \cite{FarKat10NonsepUHF1}.

\begin{thm}
\label{prp:gr_AF-alg}
Every locally finite-dimensional \ca{} has generator rank at most one.

If $A$ is a separable AF-algebra, then a generic element of $A$ is a generator.
In particular, $A$ is singly generated.
\end{thm}
\begin{proof}
By \autoref{prp:gr_fd}, every finite-dimensional \ca{} has generator rank at most one.
By definition, a locally finite-dimensional \ca{} is approximated by finite-dimensional sub-\ca{s} and thus has generator rank at most one by \autoref{prp:gr_approx_limits}.
Since every separable AF-algebra is locally finite-dimensional, we have $\gr(A)\leq 1$.
By \autoref{rmk:gr1_means_generic}, a generic element of $A$ is a generator.
\end{proof}

Recall that a von Neumann algebra $M$ is said to be \emph{hyperfinite} if it contains a directed family of finite-dimensional sub-\ca{s} $M_\lambda\subseteq M$ such that $\bigcup_\lambda M_\lambda$ is weak*-dense in $M$;
see \cite[Definition~III.3.4.1, p.291]{Bla06OpAlgs}.
It is known that every hyperfinite von Neumann algebra with separable predual is singly generated (as a von Neumann algebra);
see \cite[Theorem~1]{SuzSai63OpsGenvN}.
We can improve this result:

\begin{cor}
Let $M$ be a hyperfinite von Neumann algebra with separable predual.
Then the set of elements that generate $M$ as a von Neumann algebra is weak*-dense in $M$.
\end{cor}
\begin{proof}
The assumptions imply that $M$ contains a weak*-dense sub-\ca{} $A\subseteq M$ such that $A$ is a separable AF-algebra.
Set $G:=\{a\in A : A=C^*(a)\}$.
Note that every $a\in G$ generates $M$ as a von Neumann algebra.
By \autoref{prp:gr_AF-alg}, $G$ is norm-dense in $A$, and consequently weak*-dense in $M$.
\end{proof}


\providecommand{\etalchar}[1]{$^{#1}$}
\providecommand{\href}[2]{#2}

\end{document}